\documentclass[12pt, reqno, a4paper]{amsart}

\usepackage{subcaption}

\usepackage[section]{placeins}

\usepackage{ amssymb, amsmath, enumerate, amsfonts, amsthm, mathrsfs, url, bm, mathtools}

\usepackage{xcolor}  	
\usepackage{hyperref}
\hypersetup{
colorlinks,
   linkcolor={cyan!80!black},
   citecolor={cyan!80!black},
 urlcolor={cyan!80!black}
}

\usepackage{color}

\usepackage[margin=1in]{geometry}

\RequirePackage{doi}

\usepackage{amscd}
\usepackage{amsfonts}
\usepackage{float}
\usepackage{color}
\usepackage[
backend=biber,
style=alphabetic,
]{biblatex}
\usepackage{bookmark}

\usepackage{systeme}

\renewbibmacro{in:}{}
\DeclareFieldFormat{title}{#1}

\DeclareFieldFormat[article]{title}{\mkbibemph{#1}}       
\DeclareFieldFormat[incollection]{title}{\mkbibemph{#1}}  
\DeclareFieldFormat[book]{title}{\mkbibemph{#1}}          
\DeclareFieldFormat[incollection]{booktitle}{#1}          
\DeclareFieldFormat[article]{journaltitle}{#1}            

\AtEveryBibitem{%
  \ifentrytype{misc}{\DeclareFieldFormat{title}{\mkbibemph{#1}}}{}}

\DeclareFieldFormat{eprint:eprint}{arXiv:\href{https://arxiv.org/abs/#1}{#1}}

\usepackage{amssymb}

\newtheorem{theorem}{Theorem}[section]
\newtheorem{lemma}{Lemma}[section]

\theoremstyle{definition}

\newtheorem{corollary}{Corollary}[section]
\newtheorem{proposition}{Proposition}[section]
\newtheorem{conjecture}{Conjecture}[section]

\theoremstyle{remark}
\newtheorem{remark}{Remark}[section]

\numberwithin{equation}{section}

\newcommand{\Mod}[1]{\ (\mathrm{mod}\ #1)}

\DeclareRobustCommand{\stirling}{\genfrac\{\}{0pt}{}}

\renewcommand{\leq}{\leqslant}
\renewcommand{\geq}{\geqslant}

\begin{document}

 \title[Moments of primes in progressions]{Moments of primes in progressions to a large modulus}


\author{}
\address{}
\curraddr{}
\email{}
\thanks{}

\author{Sun-Kai Leung}
\address{D\'epartement de math\'ematiques et de statistique\\
Universit\'e de Montr\'eal\\
CP 6128 succ. Centre-Ville\\
Montr\'eal, QC H3C 3J7\\
Canada}
\curraddr{}
\email{sun.kai.leung@umontreal.ca}
\thanks{}



\date{}

\dedicatory{}


\begin{abstract}
Assuming a uniform $q$-variant of the prime $k$-tuple conjecture, we compute moments of the number of primes in arithmetic progressions to a large modulus $q$ 
as the residue classes vary. Consequently, depending on the size of $\varphi(q)$, the prime count follows either a Gaussian or a Poisson distribution. In particular, the least prime in progressions follows an exponential distribution, with some unexpected discrepancies observed for smooth moduli.
\end{abstract}

\maketitle


\section{Introduction}

The study of primes in arithmetic progressions began with Dirichlet's celebrated theorem on the infinitude of primes in 1837. It states that if $(a,q)=1,$ i.e., admissible, then there are infinitely many primes $p \equiv a \Mod{q}.$ A century later, building on Siegel's theorem on Dirichlet $L$-functions, Walfisz (1936) provided an estimate of the number of primes up to $N$ in admissible arithmetic progressions modulo $q$ provided that $q \leq (\log N)^A,$ now known as the Siegel–Walfisz theorem, i.e., if $(a,q)=1,$ then
\begin{align*}
\psi(N;q,a):=\sum_{\substack{1 \leq n \leq N\\n \equiv a \Mod{q}}} \Lambda(n) 
= \frac{N}{\varphi(q)}+O ( N \exp ( -c_A \sqrt{\log N} ) )
\end{align*}
as $N \to \infty,$ where $\Lambda$ is the von Mangoldt function.

However, due to our limited understanding of potential Siegel zeros (if they exist), primes in arithmetic progressions to large moduli remain mysterious, especially when $\log q \asymp \log N.$ In this direction, Linnik (1944) unconditionally proved the existence of absolute constants $c, L>0$ such that the least prime in an (admissible) arithmetic progression modulo $q$ is at most $ c q^L.$ Nevertheless, a natural but ambitious question arises: given a large modulus $q$, what is the distribution of primes in arithmetic progressions as $a \Mod{q}$ varies? In particular, one has to compute the variance
\begin{align*}
V(N;q):=\sideset{}{^*}\sum_{a \Mod{q}} \left( \psi(N;q,a)-\frac{N}{\varphi(q)} \right)^2.
\end{align*}

The mean of $V(N;q)$ over all moduli $q \leq Q$ has a long history, and related results are commonly referred to as the ``Barban--Davenport--Halberstam" type theorems, notably by Hooley, Montgomery and Vaughan in a series of papers. For individual $V(N;q),$ however, much less is known. Hooley \cite[p. 536]{MR506067} conjectured that $V(N;q) \sim N\log q$ in some unspecified range of $q \leq N.$ Assuming the Riemann hypothesis and a uniform Hardy--Littlewood conjecture on prime pairs, Friedlander and Goldston \cite{MR1412558} showed that the asymptotic holds for $q \in [N^{\frac{1}{2}+\epsilon}, N].$ Recently, Fiorilli \cite{MR3356760} and de la Bret\`{e}che--Fiorilli \cite{MR4611407} have respectively studied variance and moments of primes in arithmetic progressions from a Fourier-analytic perspective (under the generalized Riemann hypothesis and the linear independence conjecture).

In this paper, we revisit the ``physical side". Adapting the method of Montgomery and Soundararajan \cite{MR2104891}, we compute moments of primes in arithmetic progressions under a uniform $q$-variant of the prime $k$-tuple conjecture.


Inspired by Granville's random model (see \cite[pp. 23-24]{MR1349149}), we propose the following conjecture, which can be viewed as a $q$-variant of the uniform prime $k$-tuple conjecture of Montgomery and Soundararajan \cite[Theorem 3]{MR2104891}.

\begin{conjecture}[Uniform $q$-variant of the Hardy--Littlewood prime $k$-tuple conjecture] \label{conj:ii}

Given integers $k, q, D, N \geq 1$ for which $\varphi(q) \leq N/\log N$ and $D \leq N,$ let $\mathcal{D} \subseteq [1,D]$ be a set consisting of $k$ distinct multiples of $q$. Then for any $1 \leq x \leq N,$ there exists an absolute constant $0<\delta<1/2$ such that
\begin{align*}
\sum_{\substack{1 \leq n \leq x\\ (n,q)=1}} \prod_{d \in \mathcal{D}}\Lambda(n+d)=
\mathfrak{S}(\mathcal{D};q) 
\left| \{ 1 \leq n \leq x \,:\, (n,q)=1 \} \right|+O_{k,\delta}(N^{1-\delta}),
\end{align*}
or equivalently (by partial summation) that
\begin{align*}
\sum_{\substack{1 \leq n \leq x\\(n,q)=1}} \prod_{d \in \mathcal{D}}1_{\mathbb{P}}(n+d) =
\mathfrak{S}(\mathcal{D};q)
\sum_{\substack{1 \leq n \leq x\\(n,q)=1}} 
\left( \prod_{d \in \mathcal{D}} \log (n+d) \right)^{-1}
+O_{k,\delta}(N^{1-\delta}),
\end{align*}
where $1_{\mathbb{P}}$ is the prime indicator function, and $\mathfrak{S}(\mathcal{D};q)$ is the singular series
defined as the Euler product
\begin{align*}
\left( \frac{\varphi(q)}{q} \right)^{-k}
\prod_{p \nmid q} \left( 1-\frac{1}{p} \right)^{-k}
\left( 1-\frac{v_p(\mathcal{D})}{p} \right)
\end{align*}
with $v_p(\mathcal{D}):=|\{ d \in \mathcal{D} \Mod{p} \}|.$

\end{conjecture}
Note that $
\mathfrak{S}(\mathcal{D};q) = \left( \varphi(q)/q \right)^{-1} \mathfrak{S}(\mathcal{D}),$ where 
\begin{align*}
\mathfrak{S}(\mathcal{D}):=\prod_{p} \left( 1-\frac{1}{p} \right)^{-k}
\left( 1-\frac{v_p(\mathcal{D})}{p} \right)
\end{align*}
is the classical singular series.


Assuming the conjecture, we compute moments of weighted prime counts in progressions to a modulus $q,$ provided that $\varphi(q) \in [N^{1-\epsilon}, \epsilon\frac{N}{\log N}],$ which can be viewed as the progression counterpart of \cite[Theorem 3]{MR2104891}.

\begin{theorem} \label{thm:variesa}
Assume Conjecture \ref{conj:ii}. Let $K \geq 0$ and $N, q \geq 2$ be integers. Suppose $N^{1-2\delta/(K+3)} \leq \varphi(q) \leq \frac{N}{\log N}.$ Then as $N \to \infty,$ we have
\begin{gather*} 
M_K(N;q):=\frac{1}{\varphi(q)}\sideset{}{^*}\sum_{a \Mod{q}} 
\left( \frac{1}{\sqrt{N\log q/\varphi(q)}}
\sum_{\substack{1 \leq n \leq N\\n \equiv a \Mod{q}}}
\left(\Lambda(n)-\frac{q}{\varphi(q)}\right)
\right)^K\\
=\mu_K +O_{K,\delta}
\left( \left( \frac{N}{\varphi(q)\log N} \right)^{-\frac{1}{8K}}+ \left(  \frac{N}{\varphi(q)} \right)^{\frac{K}{2}+1}(\log N)^{K/2} N^{-\delta}+\frac{(\log \log N)^3}{\log N} \right),
\end{gather*}
where 
\begin{align*}
\mu_{K}:=
\begin{cases}
\frac{K!}{2^{K/2}(K/2)!} & \mbox{{\normalfont if $K$ is even,} }\\
\hfil 0 & \mbox{{\normalfont otherwise }} 
\end{cases}
\end{align*}
is the $K$-th standard Gaussian moment.
\end{theorem} 

Applying the method of moments,
the Gaussian law follows readily from Theorem \ref{thm:variesa}.

\begin{corollary} \label{cor:variesa}
Assume Conjecture \ref{conj:ii}. Suppose $N^{1-\epsilon}\leq \varphi(q) \leq \epsilon \frac{N}{\log N}$ for any $\epsilon>0.$ If a reduced residue $a \Mod{q}$ is chosen uniformly at random, then as $N \to \infty$, we have convergence in distribution to a standard Gaussian
\begin{align*}
\frac{1}{ \sqrt{N \log q/\varphi(q)}}
\left( \psi(N;q,a) - \frac{N}{\varphi(q)} \right)
\xrightarrow[]{d} \mathcal{N}(0,1),
\end{align*} 
i.e., for any fixed $\alpha< \beta,$ we have
\begin{gather*}
\frac{1}{\varphi(q)}\left| \left\{ a \Mod{q}\,: \, 
\frac{1}{ \sqrt{N \log q/\varphi(q)}}
\left( \psi(N;q,a) - \frac{N}{\varphi(q)} \right)
 \in (\alpha,\beta] 
\right\} \right| \\
=\frac{1}{\sqrt{2\pi}}\int_{\alpha}^{\beta}e^{-\frac{x^2}{2}}dx+o_{N \to \infty}(1).
\end{gather*}
\end{corollary}

Our corollary is complementary to a recent conjecture of
de la Bret\`{e}che and Fiorilli 
\cite
[Conjecture 1.9]
{MR4611407}
that in the range of $q \in [(\log \log N)^{1+\epsilon}, N^{1-\epsilon}],$ if $V \in \mathbb{R}$ is fixed, then
\begin{gather*}
\frac{1}{\varphi(q)}\left| \left\{ a \Mod{q}\,: \, 
\frac{1}{ \sqrt{N \log q/\varphi(q)}}
\left( \psi(N;q,a) - \frac{\psi(N,\chi_{0,q})}{\varphi(q)} \right)
>V 
\right\} \right|  \\
=\frac{1}{\sqrt{2\pi}}\int_{V}^{\infty}e^{-\frac{x^2}{2}}dx+o_{N \to \infty}(1),
\end{gather*}
where $\psi(N,\chi_{0,q}):=\sum_{1 \leq n \leq N} \Lambda(n)\chi_{0,q}(n)$ with $\chi_{0,q} \Mod{q}$ being the principal character.

Similarly, we also compute moments of the number of primes in arithmetic progressions to a large modulus $q,$ provided that $\varphi(q) \asymp \frac{N}{\log N}.$

\begin{theorem} \label{thm:poisson}
Assume Conjecture \ref{conj:ii}. Let $k \geq 1$ be an integer. Suppose $\varphi(q) \asymp \frac{N}{\log N}.$ Then
\begin{align*}
\frac{1}{\varphi(q)}\sideset{}{^*}\sum_{a \Mod{q}}
\pi(N;q,a)^k=\sum_{j=1}^k \stirling{k}{j}
\left( \frac{N}{\varphi(q)\log N} \right)^j
+O\left( \frac{\log\log N}{\log N} \right)
\end{align*}
as $N \to \infty,$ where $\stirling{k}{j}$ is the Stirling number of the second kind, i.e., the number of ways to partition a set of $k$ objects into $j$ non-empty subsets.
The implied constant depends on $k$ and the ratio $\varphi(q)/\frac{N}{\log N}$ only.
\end{theorem}

Applying the method of moments, the Poissonian law follows from Theorem \ref{thm:poisson}. 

\begin{corollary}
Assume Conjecture \ref{conj:ii}. Let $k \geq 0$ be a fixed integer. Suppose $\varphi(q)=(1+o(1))\frac{N}{\lambda \log N}$ for some fixed $\lambda>0$ as $N \to \infty.$ 
Then 
\begin{align*}
\lim_{N \to \infty} \frac{1}{\varphi(q)}\left| \left\{ a \Mod{q} \,: \, 
\pi(N;q,a)=k \right\}
 \right| 
=e^{-\lambda}\frac{\lambda^k}{k!}.
\end{align*}
\end{corollary}
How big is the first prime in a given arithmetic progression? 
Let $p(q,a)$ denote the least prime in the arithmetic progression congruent to $a$ modulo $q$. Then Linnik's theorem states that there exists an absolute constant $L>0$ such that $p(q,a) \ll q^L.$ In fact, Xylouris \cite{MR3086819} proved that one can take $L=5.$ However, we believe this is far from the full truth and Heath-Brown \cite{MR491558} conjectured that $p(q,a) \ll \varphi(q)\log^2 q.$  
Nonetheless, assuming the generalized Riemann hypothesis, Tur\'{a}n \cite{MR836} proved that $p(q,a) \ll_{\epsilon} \varphi(q)\log^{2+\epsilon} q$ except possibly for $o(\varphi(q))$ arithmetic progressions. Here, we show that as soon as $f(q) \to \infty$ with $q \to \infty,$ the least prime $p(q,a)$ is $\leq f(q)\varphi(q)\log q$ for almost all $a \Mod{q}$ under the uniform variant of the prime $k$-tuple conjecture. In fact, as a direct consequence of Theorem \ref{thm:poisson}, the least prime in arithmetic progressions to a large modulus has an exponential limiting distribution.
\begin{corollary} \label{cor:exp}
Assume Conjecture \ref{conj:ii}. Let $t > 0$ be a fixed real number. Then 
\begin{align*}
\lim_{q \to \infty} \frac{1}{\varphi(q)} \left| \{ a \Mod{q} \,:\, 
p(q,a) \leq t \varphi(q) \log q
\} \right| = 1-e^{-t}.
\end{align*}
In particular,
as $q \to \infty$, the mean of $p(q,a)$ is
\begin{align*}
\frac{1}{\varphi(q)}\sideset{}{^*}\sum_{a \Mod{q}} p(q,a) =(1+o(1))\varphi(q) \log q.
\end{align*}
\end{corollary}

\begin{remark}
Erd\H{o}s \cite[Theorem 2]{MR29941} proved unconditionally that for each fixed $t>0,$ we have
\begin{align*}
\liminf_{q \to \infty} \frac{1}{\varphi(q)} \left| \{ a \Mod{q} \,:\, 
p(q,a) \leq t \varphi(q) \log q
\} \right| >0.
\end{align*}
\end{remark}

These results can be viewed as the progression counterparts of the work by Gallagher \cite{MR409385}. See also \cite{MR1024571}, where $a$ is fixed but $q$ varies.

In Section \ref{sec:stat}, we examine some statistics on the least primes in arithmetic progressions. Finally, Section \ref{sec:discre} explains the unexpected discrepancies observed for smooth moduli.  
\\~\\
\noindent\textit{Notation.}
Throughout the paper, we use the standard big $O,$ little $o$ notations, and the Vinogradov notation $\ll,$ where the implied constants depend only on the subscripted parameters, unless otherwise specified. Also, we write $\sum_{a \Mod{q}}^{*}$ to denote the sum over reduced residues modulo $q.$ 

\section{Reduced residues in a progression to a large modulus}

Before embarking on the proof of Theorem \ref{thm:variesa}, we first compute moments of the number of reduced residues in an arithmetic progression to a large modulus. To simplify notation, all implied constants henceforth depend on $K, k, \epsilon, \delta.$

\begin{lemma}  \label{lem:vkii}
Let $k\geq 0, Q \geq 1, q \geq 2, a \geq 1$ be integers for which $(q,Q)=1$ and $(a,q)=1.$ Then for any $N \geq 2q,$ we have
\begin{align*} 
V_{k}(Q;q,a):=&\frac{1}{Q}
\sum_{ m \Mod{Q}}
\left( \sum_{\substack{1 \leq n \leq N\\n \equiv a \Mod{q}}}
\left( \left(\frac{\varphi(Q)}{Q}\right)^{-1}  1_{(n+m,Q)=1}-1\right)
\right)^{k}\\
=&
\sum_{1<r_i\mid Q}\prod_{i=1}^k\frac{\mu(r_{i})}{\varphi(r_{i})}
\sideset{}{^*}\sum_{\substack{b_{i} \Mod{r_{i}}\\ \sum_{i=1}^k b_{i}/r_{i} \equiv 0 \Mod{1}}}
\prod_{i=1}^k E_{q,a}\left( \frac{b_i}{r_i} \right),
\end{align*}
where
\begin{align*}
E_{q,a}\left( \alpha \right):=\sum_{\substack{1 \leq n \leq N\\n \equiv a \Mod{q}}} e \left( n\alpha \right).
\end{align*}

\begin{proof}
Let $c_Q(A)$ denote the Ramanujan sum $
\sideset{}{^*}\sum_{B \Mod{Q}}e\left( \frac{AB}{Q} \right).$ Then
\begin{align*}
1_{(n,Q)=1}=\frac{1}{Q}\sum_{A \Mod{Q}}c_Q(A)e\left( \frac{An}{Q} \right),
\end{align*}
so that
\begin{align} \label{eq:discretefourier}
\sum_{\substack{1 \leq n \leq N\\n \equiv a \Mod{q}}}
\left( \left(\frac{\varphi(Q)}{Q}\right)^{-1}  1_{(n+m,Q)=1}-1\right)=
\frac{1}{\varphi(Q)}\sum_{A \not\equiv 0 \Mod{Q}}c_Q(A)e\left( \frac{Am}{Q} \right)E_{q,a}\left( \frac{A}{Q} \right).
\end{align}
Recall the well-known identity (see \cite[p.44]{MR2061214})
\begin{align*}
c_Q(A)=\frac{\mu(Q/(A,Q))}{\varphi(Q/(A,Q))}\varphi(Q).
\end{align*}
Let $r=Q/(A,Q).$ Then (\ref{eq:discretefourier}) becomes
\begin{align*}
\sum_{1<r|Q}\frac{\mu(r)}{\varphi(r)}
\sideset{}{^*}\sum_{b \Mod{r}}e\left( \frac{bn}{r} \right) E_{q,a}\left( \frac{b}{r} \right).
\end{align*}
Therefore, we have
\begin{align*}
V_k(Q;q,a)&=\frac{1}{Q}\sum_{m \Mod{Q}}
\left(\sum_{1<r|Q}\frac{\mu(r)}{\varphi(r)}
\sideset{}{^*}\sum_{b \Mod{r}}e\left( \frac{bm}{r} \right) E_{q,a}\left( \frac{b}{r} \right)\right)^k\\
&= \sum_{1<r_i | Q}\prod_{i=1}^k 
\frac{\mu(r_i)}{\varphi(r_i)}\sideset{}{^*}\sum_{b_i \Mod{r_i}}
\prod_{i=1}^k E_{q,a}\left( \frac{b_i}{r_i}  \right)
\frac{1}{Q}\sum_{m \Mod{Q}} e\left( \left(\sum_{i=1}^k \frac{b_i}{r_i} \right)m\right)\\
&=
\sum_{1<r_i\mid Q}\prod_{i=1}^k\frac{\mu(r_{i})}{\varphi(r_{i})}
\sideset{}{^*}\sum_{\substack{b_{i} \Mod{r_{i}}\\ \sum_{i=1}^k b_{i}/r_{i} \equiv 0 \Mod{1}}}
\prod_{i=1}^k E_{q,a}\left( \frac{b_i}{r_i} \right)
\end{align*}
as desired.
\end{proof}

\end{lemma}

Using the lemma, we compute the moments of reduced residues in an arithmetic progression to a large modulus. This is essentially \cite[Theorem 1.1]{MR4843239} with $c_1 \equiv \cdots \equiv c_k \equiv a \Mod{r}$ (in Kuperberg's notation), except our error term depends on the modulus $q$ explicitly, which plays a key role in the proof of Theorem \ref{thm:variesa}.

\begin{proposition} \label{prop:mainii}
Let $k\geq 0, Q \geq 1, q \geq 2, a \geq 1$ be integers for which $(q,Q)=1$ and $(a,q)=1.$ Then for any $N \geq 2q$, we have
\begin{align*}
V_{k}(Q;q,a)
=\mu_k 
V_2 \left( Q;q,a \right)^{k/2}
+O \left( \left( \frac{N}{q} \right)^{\frac{k}{2}-\frac{1}{7k}} \left( \frac{\varphi(Q)}{Q} \right)^{-2^k-\frac{k}{2}} \right).
\end{align*}

\begin{proof}
We adapt the proof of \cite[Theorem 1]{MR2104891}.
Without loss of generality, let us assume $1 \leq a < q.$ Let $1<r \mid Q$ and $b \Mod{r}$ be a reduced residue. Then by definition, we have
\begin{align*}
E_{q,a}\left( \frac{b}{r}\right)&=
\sum_{\substack{1 \leq n \leq N\\n \equiv a \Mod{q}}} e \left( \frac{bn}{r} \right)\\
&=\sum_{0 \leq h \leq \frac{N-a}{q}}e\left( \frac{b}{r}(qh+a) \right)\\
&=e\left( \frac{ba}{r} \right)
\sum_{0 \leq h \leq \frac{N-a}{q}} e\left( \frac{bqh}{r} \right).
\end{align*}
Since by assumption $(q,r)=(b,r)=1,$ we also have $(bq,r)=1.$ If $c \equiv bq \Mod{r},$ then
\begin{align} \label{eq:F_q}
\left| E_{q,a}\left( \frac{b}{r}\right) \right| \leq 
F_q \left( \frac{c}{r} \right):=\min \left\{ \frac{N}{q}+1,\left\| \frac{c}{r} \right\|^{-1}\right\}.
\end{align}
Let $\mathcal{D}_k$ denote the set of $\boldsymbol{r}=(r_i)_{1 \leq i \leq k}$ for which the $r_i$'s are equal in pairs with no further equalities among them.
Then, one can argue as in the proof of \cite[Lemma 8]{MR835765} that
\begin{gather*}
\sum_{\substack{1<r_i\mid Q\\ \boldsymbol{r} \notin \mathcal{D}_k}}\prod_{i=1}^k\frac{\mu^2(r_{i})}{\varphi(r_{i})}
\sideset{}{^*}\sum_{\substack{b_{i} \Mod{r_{i}}\\ \sum_{i=1}^k b_{i}/r_{i} \equiv 0 \Mod{1}}}
\prod_{i=1}^k \left| E_{q,a}\left( \frac{b_i}{r_i}  \right) \right| \\
 \leq \sum_{\substack{1<r_i\mid Q\\ \boldsymbol{r} \notin \mathcal{D}_k}}\prod_{i=1}^k\frac{\mu^2(r_{i})}{\varphi(r_{i})}
\sideset{}{^*}\sum_{\substack{c_{i} \Mod{r_{i}}\\ \sum_{i=1}^k c_{i}/r_{i} \equiv 0 \Mod{1}}}
\prod_{i=1}^k F_{q}\left( \frac{c_i}{r_i} \right) \\
\ll
\left( \frac{N}{q} \right)^{\frac{k}{2}-\frac{1}{7k}} \left( \frac{\varphi(Q)}{Q} \right)^{-2^k-\frac{k}{2}}.
\end{gather*}
Appealing to Lemma \ref{lem:vkii}, it remains to estimate
\begin{align} \label{eq:D_kii}
\sum_{\substack{1<r_i\mid Q\\ \boldsymbol{r} \in \mathcal{D}_k}}\prod_{i=1}^k\frac{\mu(r_{i})}{\varphi(r_{i})}
\sideset{}{^*}\sum_{\substack{b_{i} \Mod{r_{i}}\\ \sum_{i=1}^k b_{i}/r_{i} \equiv 0 \Mod{1}}}
\prod_{i=1}^k E_{q,a}\left( \frac{b_i}{r_i} \right).
\end{align}
We can assume $k$ is even, for otherwise $\mathcal{D}_{\boldsymbol{k}}$ is empty. Let 
\begin{align*}
\mathcal{B}_{k}:=\{\sigma=\{\{i,j\}\}_{i,j} \,: 1 \leq i < j \leq k, |\sigma|=k/2\}
\end{align*}
denote the set of perfect matchings of $\{1,\ldots,k \}.$ Then, by definition (\ref{eq:D_kii}) is
\begin{align*}
\sum_{\sigma \in \mathcal{B}_k}
\sum_{\substack{1<r_{\tau_{\sigma}} \mid Q\\ \tau_{\sigma}=\{i,j\}\in \sigma\\ r_{\tau_{\sigma}} \text{ distinct}}}
\prod_{\tau_{\sigma} \in \sigma}
\frac{\mu^2(r_{\tau_{\sigma}})}{\varphi^2(r_{\tau_{\sigma}})}
\sideset{}{^*}\sum_{\substack{c_{\tau_{\sigma}} \Mod{r_{\tau_{\sigma}}}\\ \sum_{\tau_{\sigma}}c_{\tau_{\sigma}}/r_{\tau_{\sigma}}  \equiv 0 \Mod{1} }}
\prod_{\tau_{\sigma} \in \sigma}
J(c_{\tau_{\sigma}},r_{\tau_{\sigma}}),
\end{align*}
where
\begin{align*}
J(c,r):=
\sideset{}{^*}\sum_{\substack{ b \Mod{r}\\(c-b,q)=1}}E_{q,a}\left( \frac{b}{r} \right)
E_{q,a}\left( \frac{c-b}{q} \right).
\end{align*}
As in the proof of \cite[Theorem 1]{MR2104891}, the distinctness condition among the $r_{\tau_{\sigma}}$'s can be dropped, since the error induced is negligible (see \cite[p. 597--598]{MR2104891}). Besides, for each $\sigma \in \mathcal{B}_{},$ the main contribution comes from $c_{\tau_{\sigma}} \equiv 0 \Mod{q_{\tau_{\sigma}}}$ for all $\tau_{\sigma},$ and the rest is also negligible (see \cite[p. 598--599]{MR2104891}). In this case, note that
\begin{align*}
J(0,r)=\sideset{}{^*}\sum_{b \Mod{r}} \left|E_{q,a}\left( \frac{b}{r} \right) \right|^2.
\end{align*}
Therefore, we are left with
\begin{align*}
\sum_{\sigma \in \mathcal{B}_{k}}
\sum_{\substack{1<r_{\tau_{\sigma}} \mid Q\\ \tau_{\sigma}=\{i,j\}\in \sigma}}
\prod_{\tau_{\sigma} \in \sigma}
\frac{\mu^2(q_{\tau_{\sigma}})}{\varphi^2(q_{\tau_{\sigma}})}
\prod_{\tau_{\sigma} \in \sigma}
J_{i_1,i_2}(0,q_{\tau_{\sigma}}),
\end{align*}
which is by Lemma \ref{lem:vkii} 
\begin{align*}
 \mu_k V_2 \left( Q;q,a \right)^{k/2},
\end{align*}
and hence the proposition follows.
\end{proof}

\end{proposition}

Therefore, it remains to compute $V_2(Q;q,a).$ See \cite[Lemma 3.2]{MR4843239} for an estimate without uniformity in the modulus $q$.

\begin{lemma} \label{lem:expsumii}
Given integers $N, q \geq 2$ for which $N \geq 2q$ and $\varphi(q) \leq \frac{N}{\log N},$ let $Q\geq 1$ be an integer which is divisible by all primes $p \leq \left( N/q\right)^2$ coprime to $q.$ Then for any $(a,q)=1,$ we have
\begin{align*}
V_2(Q;q,a)= H\sum_{1<r \mid Q} \frac{\mu^2(r)}{\varphi(r)}
-\frac{\varphi(q)}{q} \times \frac{N}{q}\log \frac{N}{q}
+O\left( \frac{\varphi(q)}{q} \times \frac{N}{q} (\log \log q)^3 \right),
\end{align*}
where $H=H_{N;q,a}:=|\{ 1 \leq n \leq N \,: \, n \equiv a \Mod{q} \}|.$

\begin{proof}
By definition, we have
\begin{align*}
\left| E_{q,a}(\alpha) \right|^2 &= \left| \sum_{\substack{1 \leq n \leq N\\n \equiv a \Mod{q}}} e(n \alpha) \right|^2 \\
&=\sum_{\substack{1 \leq n_1, n_2 \leq N\\n_1 \equiv n_2 \equiv a \Mod{q}}}
e((n_1-n_2)\alpha).
\end{align*}
Let $m=(n_1-n_2)/q.$ Then this becomes
\begin{align*}
\sum_{0 \leq |m| \leq  H }
\left(   H-|m| \right) e\left( qm\alpha \right).
\end{align*}
Therefore, we obtain
\begin{align*}
\sideset{}{^*}\sum_{b \Mod{r}} \left| E_{q,a} \left( \frac{b}{r} \right) \right|^2
=\sum_{0 \leq |m| \leq  H }
\left(  H -|m| \right)
c_r(qm),
\end{align*}
where $c_r(qm)$ is Ramanujan's sum. Since $c_r(0)=1$ and $c_r(-qm)=c_r(qm),$ this is
\begin{align*}
\varphi(r) H
+2 \sum_{1 \leq m \leq H}
\left( H - m \right) c_r(qm),
\end{align*}
so that by Lemma \ref{lem:vkii}
\begin{align} \label{eq:v2Qqa}
V_2(Q;q,a)=
H \sum_{1<r \mid Q} \frac{\mu^2(r)}{\varphi(r)}
+2\sum_{1 \leq m \leq H}
\left( H - m \right)
\sum_{1<r \mid Q} \frac{\mu^2(r)}{\varphi^2(r)}c_r(qm)
.
\end{align}
Using the multiplicity of Ramanujan's sum in $r$, we have
\begin{align*}
\sum_{r \mid Q} \frac{\mu^2(r)}{\varphi^2(r)}c_r(qm)
&=\prod_{\substack{p \mid Q\\p \mid qm}} \left( 1+\frac{1}{p-1} \right)\prod_{\substack{p \mid Q\\p \nmid qm}} \left( 1-\frac{1}{(p-1)^2} \right)
 \\
&= \frac{\varphi(q)}{q} 
\prod_{\substack{p \mid qQ\\p \mid qm}} \left( 1+\frac{1}{p-1} \right)
\prod_{\substack{p \mid qQ\\p \nmid qm}} \left( 1-\frac{1}{(p-1)^2} \right)
.
\end{align*}
Since by assumption $qQ$ is divisible by all primes $p \leq (N/q)^2,$ using the following Mertens-type estimates
\begin{align*}
\prod_{p | m} \left( 1+\frac{1}{p-1} \right) \ll \log m
\end{align*}
and
\begin{align*}
\prod_{p>(N/q)^2} \left( 1 - \frac{1}{(p-1)^2}\right)^{-1} =& 1+O \left(
\left( \frac{N}{q} \right)^{-2} \left(\log \frac{N}{q} \right)^{-1} \right)
\\
=& 1+O \left( H^{-2} (\log H)^{-1} \right),
\end{align*}
this is
\begin{align} \label{eq:qQet}
\frac{\varphi(q)}{q} 
\prod_{\substack{p \mid qm}} \left( 1+\frac{1}{p-1} \right)
\prod_{\substack{p \nmid qm \\ }} \left( 1-\frac{1}{(p-1)^2} \right)
+O \left(  H^{-2} \right).
\end{align}
The main term here is
\begin{align} \label{eq:qQmt}
 \frac{\varphi(q)}{q} 
\prod_{\substack{p \mid qm}} \left( 1-\frac{1}{p} \right)^{-2}   \left( 1-\frac{1}{p} \right)
\prod_{\substack{p \nmid qm \\ }} \left( 1-\frac{1}{p} \right)^{-2}   \left( 1-\frac{2}{p} \right)= \frac{\varphi(q)}{q} \times \mathfrak{S}(qm),
\end{align}
where $\mathfrak{S}(qm):=\mathfrak{S}(\{0,qm\})$ for convenience.
Appealing to \cite[Proposition 3]{MR1412558}, we have
\begin{align} \label{eq:sing}
\frac{\varphi(q)}{q} \times 2\sum_{1 \leq m \leq H}(H-m) \mathfrak{S}(qm) 
=H^2+\frac{\varphi(q)}{q} \left( - H\log H+ B_q H+I_{\delta}(H,q)\right)
\end{align}
for any $0 < \delta <\frac{1}{2},$ where 
$B_q:=1-\gamma-\log 2\pi-\sum_{p \mid q}\frac{\log p}{p-1}$ and
$I_{\delta}(H,q)$ is an integral 
\begin{align*}
\ll \mathfrak{S}(2q) H^{1-\delta} \delta^{-1} \left( \frac{1}{2}-\delta \right)^{-\frac{3}{2}}\prod_{p \mid q} \left( 1+\frac{1}{p^{1-\delta}} \right)
\left( 1+\frac{1}{p^{2(1-\delta)}} \right).
\end{align*}
Taking $\delta=\frac{1}{\log \log q},$ this is
$\ll H(\log \log q)^3$ (see \cite[p. 322]{MR1412558}), so that (\ref{eq:sing}) becomes
\begin{align*}
H^2+\frac{\varphi(q)}{q} \left( - H\log H+ B_q H +O( H(\log \log q)^3)\right).
\end{align*}
Therefore, combining with (\ref{eq:qQet}) and (\ref{eq:qQmt}), the expression (\ref{eq:v2Qqa}) is
\begin{gather*}
H \sum_{1<r \mid Q} \frac{\mu^2(r)}{\varphi(r)}    
+2\sum_{1 \leq m \leq H}
\left( H - m \right)
\left(\sum_{r \mid Q} \frac{\mu^2(r)}{\varphi^2(r)}c_r(qm)-1 \right) \\
= H \sum_{1<r \mid Q} \frac{\mu^2(r)}{\varphi(r)}   
+\frac{\varphi(q)}{q} \left( - H\log H+ A_q H+O(H(\log \log q)^3)\right)
+O(1),
\end{gather*}
where 
\begin{align*}
A_q:=B_q+\frac{q}{\varphi(q)}=1-\gamma-\log 2\pi-\sum_{p \mid q}\frac{\log p}{p-1}+\frac{q}{\varphi(q)}.
\end{align*}
Since 
$H \gg N/q \gg \frac{\log q}{\log \log q},$
$\sum_{p \mid q}\frac{\log p}{p-1} \ll \log\log q,$ and $\frac{q}{\varphi(q)} \ll \log\log q,$ the lemma follows.
\end{proof}

\end{lemma}

\section{Properties of singular series along progressions}

Given $\mathcal{D}=\{ d_1, \ldots, d_k\}$ and $q \geq 1,$ let $\mathfrak{S}_0\left( \mathcal{D};q \right)$ denote the normalized singular series
\begin{align*}
\left(\frac{\varphi(q)}{q}\right)^{-k}
\sum_{\substack{r_i>1\\(r_i,q)=1}}\prod_{i=1}^k\frac{\mu(r_{i})}{\varphi(r_{i})}
\sideset{}{^*}\sum_{\substack{b_{i} \Mod{r_{i}}\\ \sum_{i=1}^k b_{i}/r_{i} \equiv 0 \Mod{1}}}
e\left( \sum_{i=1}^k \frac{b_id_i}{r_i} \right).
\end{align*}
We are interested in such singular series due to the following lemma.
\begin{lemma} \label{lem:afterconjii}
Assume Conjecture \ref{conj:ii}. Given integers $k, q, D, N \geq 1$ for which $\varphi(q) \leq N/\log N$ and $D \leq N,$ let $\mathcal{D} \subseteq [1,D]$ be a set consisting of $k$ distinct multiples of $q$. Then for any $1 \leq x \leq N,$ we have
\begin{align*}
\sum_{\substack{1 \leq n \leq x\\(n,q)=1}} \prod_{d \in \mathcal{D}}\Lambda_0(n+d;q)=
\mathfrak{S}_0(\mathcal{D};q)\left| \{ 1 \leq n \leq x \,:\, (n,q)=1 \} \right| 
+O\left( \left( \frac{\varphi(q)}{q} \right)^{-k}  N^{1-\delta} \right),
\end{align*}
where
$\Lambda_0(n;q):=\Lambda(n)-\frac{q}{\varphi(q)}$ for $n \geq 1.$

\begin{proof}
Appealing to \cite[Lemma 3]{MR2104891} gives
\begin{align*}
\mathfrak{S}(\mathcal{D};q)
=\left(\frac{\varphi(q)}{q}\right)^{-k}
\sum_{\substack{r_i \geq 1\\(r_i,q)=1}}\prod_{i=1}^k\frac{\mu(r_{i})}{\varphi(r_{i})}
\sideset{}{^*}\sum_{\substack{b_{i} \Mod{r_{i}}\\ \sum_{i=1}^k b_{i}/r_{i} \equiv 0 \Mod{1}}}
e\left( \sum_{i=1}^k \frac{b_id_i}{r_i} \right),
\end{align*}
so that
\begin{align} \label{eq:altii}
\mathfrak{S}_0(\mathcal{D};q)
=\sum_{\mathcal{J} \subseteq \mathcal{D}}
\left( -\frac{q}{\varphi(q)} \right)^{k-|\mathcal{J}|} \mathfrak{S}(\mathcal{J};q).
\end{align}
Therefore, we have
\begin{align*}
\sum_{\substack{1 \leq n \leq x\\(n,q)=1}} \prod_{d \in \mathcal{D}}\Lambda_0(n+d;q) 
=& \sum_{\substack{1 \leq n \leq x\\(n,q)=1}} \prod_{d \in \mathcal{D}} \left(\Lambda(n+d;q) -\frac{q}{\varphi(q)} \right) \\
=& \sum_{\mathcal{J} \subseteq \mathcal{D}}
\left( -\frac{q}{\varphi(q)} \right)^{k-|\mathcal{J}|} 
\sum_{\substack{1 \leq n \leq x\\(n,q)=1}}
\prod_{d \in \mathcal{J}} \Lambda(n+d;q),
\end{align*}
which is
\begin{align*}
\sum_{\mathcal{J} \subseteq \mathcal{D}}
\left( -\frac{q}{\varphi(q)} \right)^{k-|\mathcal{J}|}
\mathfrak{S}(\mathcal{J};q)  \times
\left| \{ 1 \leq n \leq x \,:\, (n,q)=1 \} \right|
+O\left( \left( \frac{q}{\varphi(q)} \right)^{k}  N^{1-\delta} \right).
\end{align*}
under Conjecture \ref{conj:ii}. Then, the lemma follows from (\ref{eq:altii}).
\end{proof}

\end{lemma}

We also need a lemma on the mean of the singular series along an arithmetic progression. This is essentially \cite[Theorem 1.2]{MR4843239} with $c_1 \equiv \cdots \equiv c_k \equiv a \Mod{r}$ (in Kuperberg's notation), except our error term again depends on the modulus explicitly, which is crucial to the proof of Theorem \ref{thm:variesa}.

\begin{lemma} \label{lem:sumofsingii}
Given integers $k \geq 1, N, q \geq 2$ for which $N \geq 2q$ and $\varphi(q) \leq \frac{N}{\log N},$ let
$R_k(N;q,a)$ denote the sum
\begin{align*}
\sum_{\substack{1 \leq n_{i} \leq N
\\ n_{i} \equiv a \Mod{q}\\ n_i \text{ \normalfont distinct}}}
\mathfrak{S}_0\left( \mathcal{N};q \right),
\end{align*}
where $\mathcal{N}:=\{ n_1,\ldots,n_k \}.$
Then 
\begin{align*}
R_k(N;q,a)=\mu_k 
\left( \left(\frac{\varphi(q)}{q}\right)^{-2} \left(
V_2(Q;q,a)- H \sum_{\substack{1<r \mid Q }} \frac{\mu^2 (r)}{\varphi(r)}
\right)\right)^{\frac{k}{2}}+O \left( \left(\frac{N}{q}\right)^{\frac{k}{2}-\frac{1}{7k}+\epsilon} \right),
\end{align*}
where 
$H=H_{N;q,a}:=|\{ 1 \leq n \leq N \,: \, n \equiv a \Mod{q} \}|.$

\begin{proof}
We adapt the proof of \cite[Theorem 2]{MR2104891}. Let $y \geq N/q.$ Then by definition
\begin{align*} 
\mathfrak{S}\left( \mathcal{N};q \right)
=\left( \frac{\varphi(q)}{q} \right)^{-k}
\prod_{\substack{p \leq y \\ p \nmid q}} \left( 1-\frac{1}{p} \right)^{-k}
\left( 1-\frac{v_p(\mathcal{N})}{p} \right)
\prod_{\substack{p > y \\ p \nmid q}} \left( 1-\frac{1}{p} \right)^{-k}
\left( 1-\frac{v_p(\mathcal{N})}{p} \right).
\end{align*}
Since $n_i \equiv a \Mod{q}$ for any $i$ and are distinct, we must have $v_p(\mathcal{N})=k$ if $p>y$ and $p \nmid q,$ so that the last product is 
\begin{align*}
\prod_{\substack{p > y \\ p \nmid q}} \left( 1-\frac{1}{p} \right)^{-k}
\left( 1-\frac{v_p(\mathcal{N})}{p} \right)
&=\prod_{\substack{p > y \\ p \nmid q}} \left( 1+O\left( \frac{1}{ p^2 } \right) \right)\\
&=1+O\left( \frac{1}{y\log y} \right).
\end{align*}
Also, since the first product is $\ll (\log y)^{k-1},$ we have
\begin{align} \label{eq:usefulater}
\mathfrak{S}(\mathcal{N};q) \ll (\varphi(q)/q)^{-k}(\log y)^{k-1},
\end{align}
or more precisely
\begin{align*}
\left( \frac{\varphi(q)}{q} \right)^{-k}
\prod_{\substack{p \leq y \\ p \nmid q}} \left( 1-\frac{1}{p} \right)^{-k}
\left( 1-\frac{v_p(\mathcal{N})}{p} \right)
+O
\left( \left(\frac{\varphi(q)}{q}\right)^{-k} \times \frac{(\log y)^{k-2}}{y} \right),
\end{align*}
which is by \cite[Lemma 3]{MR2104891}
\begin{gather*}
\left(\frac{\varphi(q)}{q}\right)^{-k}
\sum_{\substack{r_i \geq 1\\(r_i,q)=1\\p \mid r \implies p \leq y}}\prod_{i=1}^k\frac{\mu(r_{i})}{\varphi(r_{i})}
\sideset{}{^*}\sum_{\substack{b_{i} \Mod{r_{i}}\\ \sum_{i=1}^k b_{i}/r_{i} \equiv 0 \Mod{1}}}
e\left( \sum_{i=1}^k \frac{b_in_i}{r_i} \right) \\
+O
\left( \left(\frac{\varphi(q)}{q}\right)^{-k} \times \frac{(\log y)^{k-2}}{y} \right).
\end{gather*}
Let $y=(N/q)^{k+1}$ and $Q=\prod_{\substack{p \leq y, p \nmid q}}p.$ Then
\begin{align*}
R_k(N;q,a)=\left(\frac{\varphi(q)}{q}\right)^{-k}
\sum_{\substack{1<r_i \mid Q}}
\prod_{i=1}^k \frac{\mu(r_i)}{\varphi(r_i)}
S(r_1,\ldots,r_k; q,a)+O \left( \left(\frac{\varphi(q)}{q}\right)^{-k} \right),
\end{align*}
where
\begin{align*}
S(r_1,\ldots,r_k; q,a):=
\sum_{\substack{1 \leq n_{i} \leq N
\\ n_{i} \equiv a \Mod{q}\\ n_i \text{ \normalfont distinct}}} \,
\sideset{}{^*}\sum_{\substack{b_{i} \Mod{r_{i}}\\ \sum_{i=1}^k b_{i}/r_{i} \equiv 0 \Mod{1}}}
e\left( \sum_{i=1}^k \frac{b_in_i}{r_i} \right).
\end{align*}

To apply Proposition \ref{prop:mainii}, let us
remove the distinctness condition on the $n_i$'s. Denote
\begin{align*}
\delta_{i,j}:=
\begin{cases}
1 &\mbox{if $d_i = d_j,$} \\
0 &\mbox{otherwise.}
\end{cases}
\end{align*}
Then
\begin{align*}
\prod_{1 \leq i<j \leq k}\left(1-\delta_{i,j}\right)&=
\begin{cases}
 1 & \mbox{{\normalfont if $d_{i}$'s are distinct,} }\\
 0 & \mbox{{\normalfont otherwise. }} 
\end{cases}
\end{align*}
When the left-hand side above is expanded, we obtain a linear combination of products
of the $\delta$ symbols. Let $\Delta$ denote such a product, and $|\Delta|$ the number of factors in the
product. As in the proof of \cite[Theorem 2]{MR2104891}, we define an equivalence relation on these $\delta$-products by setting $\Delta_1 \sim \Delta_2$ if
$\Delta_{1}$ and $\Delta_2$ have the same value for all choices of $d_i$'s. Given a partition $\mathbb{P}_i=\{ \mathcal{S}_{1}, \ldots, \mathcal{S}_{M} \}$ of the set $\{ 1,\ldots,k_i \},$ let
\begin{align*}
\Delta_{\mathbb{P}}:=\prod_{m=1}^{M} 
\prod_{\substack{ i<j \\ i, j \in \mathcal{S}_m }}
\delta_{i,j}.
\end{align*}
Note that every equivalence class of $\delta$-products contains a unique $\Delta_{\mathbb{P}}$. Thus, we
have a bijective correspondence between equivalence classes of $\delta$-products and partitions
of $\{ 1,\ldots,k \}.$ For a partition $\mathbb{P}$, denote
\begin{align*}
w( \mathbb{P})
:=
\sum_{\Delta \sim \Delta_{\mathbb{P}}}
(-1)^{|\Delta|},
\end{align*}
so that
\begin{align*}
\prod_{1 \leq i<j \leq k}(1-\delta_{i,j})
=
\sum_{\mathbb{P}} w(\mathbb{P}) \Delta_{\mathbb{P}}.
\end{align*}
It follows that
\begin{align*}
S( r_1,\ldots,r_k; q, a )=
\sum_{\substack{\mathbb{P}=\{ \mathcal{S}_{1}, \ldots, \mathcal{S}_{M} \}}} w(\mathbb{P}) 
\sideset{}{^*}\sum_{\substack{b_{i} \Mod{r_{i}}\\ \sum_{i=1}^k b_{i}/r_{i} \equiv 0 \Mod{1}}}
 \prod_{\substack{m=1}}^{M} E_{q,a} 
\left( \sum_{i \in \mathcal{S}_{m}} \frac{b_i}{r_i} \right).
\end{align*}
As in \cite[p. 605--606]{MR2104891}, using (\ref{eq:F_q}) and the trivial bound
\begin{align*}
\left|  E_{q,a} 
\left( \sum_{i \in \mathcal{S}_{m}} \frac{b_i}{r_i} \right) \right|
\leq \frac{N}{q}+1, 
\end{align*}
the contribution to $R_{k}(N;q,a)$ from terms where $|S_{m}| > 2$ for some $m$ is
\begin{align*}
\ll (\varphi(q)/q)^{-k}\left( N/q \right)^{\frac{k-1}{2}+\epsilon},
\end{align*}
so that $R_{k}(N;q,a)$
\begin{gather} 
=\left(\frac{\varphi(q)}{q}\right)^{-k}
\sum_{\substack{\mathbb{P}=\{ \mathcal{S}_{1}, \ldots, \mathcal{S}_{M} \}\\|S_{m}| \leq 2}} w(\mathbb{P})  
\sum_{\substack{1<r_i \mid Q }}
\prod_{i=1}^k \frac{\mu(r_i)}{\varphi(r_i)} 
\sideset{}{^*}\sum_{\substack{b_i \Mod{r_i}\\ \sum_{i=1}^k b_i/r_i \equiv \, 0 \Mod{1}}}
 \prod_{m=1}^{M} E_{q,a} 
\left( \sum_{i \in \mathcal{S}_{m}} \frac{b_i}{r_i} \right) \nonumber \\
+ O \left( \left(\frac{\varphi(q)}{q}\right)^{-k} \left( \frac{N}{q} \right)^{\frac{k-1}{2}+\epsilon} \right). \label{eq:rkii}
\end{gather}

 Suppose that the partition $\mathbb{P}$ consists of $l$ doubleton sets $\mathcal{S}_{1}, \ldots, \mathcal{S}_{l}$ and $k-2l$ singleton sets $\mathcal{S}_{l+1}, \ldots, \mathcal{S}_{k-l}$.
Since no other $\delta$-product is equivalent to $\Delta_{\mathbb{P}}$, and $|\mathbb{P}| = l$, so $w(\mathbb{P}) = (-1)^{l}.$ Therefore, the term in $R_{k}(N;q,a)$ corresponding to 
such partition $\mathbb{P}$ is
\begin{align} \label{eq:p1p2ii}
(-1)^{l} \left(\frac{\varphi(q)}{q}\right)^{-k} \sum_{\substack{1<r_i \mid Q}} \prod_{i=1}^k \frac{\mu(r_i)}{\varphi(r_i)}
\sideset{}{^*}\sum_{\substack{b_i \Mod{r_i}\\ \sum_{i=1}^k b_i/r_i \equiv \, 0 \Mod{1}}} 
 \prod_{m=1}^{l} E_{q,a} 
\left( \sum_{i \in \mathcal{S}_{m}} \frac{b_i}{r_i} \right)
\prod_{m=l+1}^{k-l} E_{q,a} 
\left(  \frac{b_{\mathcal{S}_{m}}}{q_{\mathcal{S}_{m}}} \right),
\end{align}
where we are slightly abusing notation in the final product by identifying the singletons
with their unique element as in \cite[p. 66]{MR4843239}.

For $1 \leq m \leq l,$ let $c_{m}$ and $t_{m}$ be defined by the relations
\begin{align*}
\frac{c_{m}}{t_{m}} \equiv \sum_{i \in \mathcal{S}_{m}} \frac{b_i}{r_i} \Mod{1}, \qquad 1 \leq c_{m} \leq t_{m}, \qquad ( c_{m},  t_{m})=1,
\end{align*}
and denote
\begin{align*}
G \left( \frac{c}{t} \right):=
E_{q,a}\left( \frac{c}{t} \right)
\sum_{\substack{1<r_1, r_2 \mid Q } }\frac{\mu(r_1)\mu(r_2)}{\varphi(r_1)\varphi(r_2)}
\sideset{}{^*}\sum_{\substack{d_i \Mod{r_i}, i=1,2\\ \frac{d_1}{r_1}+\frac{d_2}{r_2} \equiv \, \frac{c}{t} \Mod{1}}} 1.
\end{align*}
Then, (\ref{eq:p1p2ii}) becomes
\begin{gather}
(-1)^{l} \left(\frac{\varphi(q)}{q}\right)^{-k}  \sum_{\substack{t_m  \mid Q\\ (t_m,q)=1 \\ 1 \leq m \leq l}}
\sideset{}{^*}\sum_{\substack{c_{m} \Mod{t_{m}}\\ 1 \leq m \leq l }}\prod_{\substack{m=1}}^{l}
G \left( \frac{c_{m}}{t_{m}} \right) \nonumber\\
\times 
\sum_{\substack{1<r_{m} \mid Q\\  l<m \leq k-l}} \quad
\sideset{}{^*}\sum_{\substack{b_m \Mod{r_{m}}\\ l<m \leq k-l\\\sum_{m}\frac{b_{m}}{r_{m}}+
\sum_{m}\frac{c_{m}}{t_{m}}
\equiv \, 0 \Mod{1}}}
\prod_{\substack{m=l+1}}^{k-l}
\frac{\mu(r_{m})}{\varphi(r_{m})}
E_{q,a}\left( \frac{b_{m}}{r_{m}} \right).   \label{eq:Gii}
\end{gather}
As in \cite[p. 607--609]{MR2104891}, the main contribution 
comes from those $m$'s for which $r_{m}=1,$ and the contribution of the remaining terms is again $\ll  (\varphi(q)/q)^{-k} \left( N/q \right)^{\frac{k-1}{2}+\epsilon}.$
Also, since
\begin{align*}
G(1)=H \sum_{\substack{1<r \mid Q }} \frac{\mu^2 (r)}{\varphi(r)},
\end{align*}
where $H:=|\{ 1 \leq n \leq N\,: \, n \equiv a \Mod{q} \}|,$ the expression (\ref{eq:Gii}) is
\begin{align*}
\left(\frac{\varphi(q)}{q}\right)^{-k} \left( - H \sum_{\substack{1<r \mid Q }} \frac{\mu^2 (r)}{\varphi(r)} \right)^{l}
V_{k-2l}(Q;q,a)
+O \left( \left(\frac{\varphi(q)}{q}\right)^{-k} \left( \frac{N}{q} \right)^{\frac{k-1}{2}+\epsilon} \right),
\end{align*}
so that (\ref{eq:rkii}) becomes
\begin{gather*}
\left(\frac{\varphi(q)}{q}\right)^{-k} \sum_{\substack{0 \leq l \leq k/2}}  
\left(  -H \sum_{\substack{1<r \mid Q }} \frac{\mu^2 (r)}{\varphi(r)} \right)^{l}
V_{k-2l}(Q;q,a)
\sum_{\substack{\mathbb{P}=\{ \mathcal{S}_{1}, \ldots, \mathcal{S}_{k-l} \}\\|\mathcal{S}_{m}| = 2, 1 \leq m \leq l\\ |\mathcal{S}_{m}| = 1, l < m \leq k-l}} 1 \\
+O \left( \left(\frac{\varphi(q)}{q}\right)^{-k} \left( \frac{N}{q} \right)^{\frac{k-1}{2}+\epsilon} \right) \\
= \left(\frac{\varphi(q)}{q}\right)^{-k} \sum_{\substack{0 \leq l \leq k/2}}
 {k \choose 2l} \mu_{l} \times
\left( - H \sum_{\substack{1<r \mid Q }} \frac{\mu^2 (r)}{\varphi(r)} \right)^{l}
V_{k-2l}(Q;q,a) \\
+O \left( \left(\frac{\varphi(q)}{q}\right)^{-k} \left( \frac{N}{q} \right)^{\frac{k-1}{2}+\epsilon} \right).
\end{gather*}
Appealing to Proposition \ref{prop:mainii}, this is
\begin{gather*}
\left(\frac{\varphi(q)}{q}\right)^{-k} \sum_{\substack{0 \leq l \leq k/2}}
 \mu_{2l} \mu_{k-2l}  {k \choose 2l} \times
\left( -H \sum_{\substack{1<r \mid Q }} \frac{\mu^2 (r)}{\varphi(r)} \right)^{l}
V_{2}(Q;q,a)^{\frac{k}{2}-l} \\
+O \left( \left(\frac{\varphi(q)}{q}\right)^{-k} \left( \frac{N}{q} \right)^{\frac{k}{2}-\frac{1}{7k}} \left( \frac{\varphi(Q)}{Q} \right)^{-2^k-\frac{k}{2}} \right),
\end{gather*}
where we use the trivial bound 
\begin{align*}
\sum_{\substack{1<r \mid Q }} \frac{\mu^2 (r)}{\varphi(r)} \ll_k \log \frac{N}{q}. 
\end{align*}
Since by assumption $N/q \gg \log q/\log \log q$ and $Q=\prod_{p \leq y, p \nmid q}p,$
it follows that the error term is $\ll \left( N/q\right)^{\frac{k}{2}-\frac{1}{7k}+\epsilon}.$
Also, since $\mu_{2l} \mu_{k-2l}{k \choose 2l}=\mu_{k}
{k \choose l},$
the main term can be simplified to
\begin{align*}
\left(\frac{\varphi(q)}{q}\right)^{-k} \mu_k
\left( 
V_2(Q;q,a)- H \sum_{\substack{1<r \mid Q }} \frac{\mu^2 (r)}{\varphi(r)}
\right)^{\frac{k}{2}}
\end{align*}
and the lemma follows.
\end{proof}



\end{lemma}

\section{Proof of Theorem \ref{thm:variesa}}
We shall adapt the proof of \cite[Theorem 3]{MR2104891}. When $K=0,$ it is clear that 
$M_{K}(N;q)=1$ holds unconditionally. 

When $K=1,$ we have
\begin{align*}
M_{1}(N;q)&=\frac{1}{\varphi(q)}
\sideset{}{^*}\sum_{a \Mod{q}} \frac{1}{\sqrt{N\log q/\varphi(q)}} \sum_{\substack{1 \leq n \leq N\\n \equiv a \Mod{q}}} \left( \Lambda(n)-\frac{q}{\varphi(q)} \right) \\
&=\frac{1}{\sqrt{\varphi(q) N \log q }}
\left( \sum_{\substack{1 \leq n \leq N\\(n,q)=1 }} \Lambda(n)-\frac{q}{\varphi(q)} \left| \{ 1 \leq n \leq N \,:\, (n,q)=1 \} \right|   \right),
\end{align*}
which is $\ll N^{-\delta} \sqrt{N/\varphi(q)\log q}$ under Conjecture \ref{conj:ii}. 

Therefore, we can assume $K \geq 2.$ Recall $\Lambda_0(n;q):=\Lambda(n)-\frac{q}{\varphi(q)},$ so that
\begin{align*}
\sideset{}{^*}\sum_{a \Mod{q}}
\left( \sum_{\substack{1 \leq n \leq N\\n\equiv a \Mod{q}}} \left(\Lambda(n)-\frac{q}{\varphi(q)}  \right) \right)^K
=\sideset{}{^*}\sum_{a \Mod{q}}
\sum_{\substack{1 \leq n_i\leq N\\n_i \equiv a \Mod{q}}}\prod_{i=1}^K\Lambda_0(n_i;q).
\end{align*}
 We group the distinct values of $n_{i}$ according to their multiplicities $M_{i}.$ Then, this becomes
\begin{align} \label{eq:distinctingii}
\sum_{k=1}^K \sum_{\substack{M_i \geq 1\\ \sum_{i=1}^k M_i=K}} {K \choose M_1,\ldots,M_k}
\sideset{}{^*}\sum_{a \Mod{q}}
\sum_{\substack{1 \leq n_1<\cdots<n_k \leq N\\n_i \equiv a \Mod{q}}}\prod_{i=1}^k \Lambda_0(n_i;q)^{M_i}.
\end{align}
Denote $\Lambda_m(n;q):=\Lambda(n)^m\Lambda_0(n;q).$ Then the binomial theorem gives
\begin{align*}
\Lambda_0(n;q)^M&=\left( \Lambda(n)-\frac{q}{\varphi(q)} \right)^{M-1}\Lambda_0(n;q)\\
&=\sum_{m=0}^{M-1}\left(-\frac{q}{\varphi(q)}\right)^{M-1-m}{M-1 \choose m}\Lambda_m(n;q).
\end{align*}
Inserting into (\ref{eq:distinctingii}), we obtain
\begin{align} \label{eq:b4deflkii}
 \sum_{k=1}^K \sum_{\substack{M_i \geq 1\\ \sum_{i=1}^k M_i=K}} {K \choose M_1,\ldots,M_k} 
 \sum_{0 \leq m_i <M_i}\prod_{i=1}^k 
 \left( - \frac{q}{\varphi(q)} \right)^{M_i-1-m_i}
 {M_i-1 \choose m_i}L_k(\boldsymbol{m};q),
\end{align}
where
\begin{align*}
L_k(\boldsymbol{m};q):=\sideset{}{^*}\sum_{a \Mod{q}}\sum_{\substack{1 \leq n_1 <\cdots<n_k \leq N\\n_i \equiv a \Mod{q}}}
\prod_{i=1}^k \Lambda_{m_i}(n_i;q).
\end{align*}
Making the change of variables $n_i=n+d_i$ for $i=1,\ldots,k,$ this becomes
\begin{align} \label{eq:aftercovii}
 \sum_{\substack{0=d_1<\cdots<d_k<N\\d_i \equiv 0 \Mod{q}}}
\sum_{\substack{1 \leq n \leq N-d_k\\(n,q)=1}}
\prod_{i=1}^k \Lambda_{m_i}(n+d_i;q).
\end{align}

To estimate $L_k(\boldsymbol{m};q),$ let us denote $\mathcal{K}:=\{1,\ldots,k \},$ and introduce the following notation:
\begin{align*}
\mathcal{H}:=\{ i \in \mathcal{K} \,:\, m_i \geq 1 \}, \quad &h=|\mathcal{H}|; \\
\mathcal{I}:=\{ i \in \mathcal{K} \,:\, m_i=0 \},
\quad &k-h=|\mathcal{I}|; \\
\mathcal{J} \subseteq \mathcal{K}, \quad &j=|\mathcal{J}|.
\end{align*}
Let $\boldsymbol{m}=(m_1,\ldots,m_k)$ and $\boldsymbol{d}=(d_1,\ldots,d_k)$ be fixed for the moment. Then
\begin{align*}
\prod_{i \in \mathcal{I}} \Lambda_0(n+d_i;q) 
\prod_{i \in \mathcal{H}}\Lambda(n+d_i;q)
&=\prod_{i \in \mathcal{I}} \Lambda_0(n+d_i;q) 
\prod_{i \in \mathcal{H}} \left(\Lambda_0(n+d_i;q) + \frac{q}{\varphi(q)} \right) \\
&= \sum_{\mathcal{I} \subseteq \mathcal{J} \subseteq \mathcal{K}} 
\left( \frac{q}{\varphi(q)} \right)^{k-j}
\prod_{i \in \mathcal{J}} \Lambda_0(n+d_i;q).
\end{align*}
Let $1 \leq x \leq N.$ Assume Conjecture \ref{conj:ii}. Then, it follows from Lemma \ref{lem:afterconjii} that
\begin{gather*}
\sum_{\substack{1 \leq n \leq x \\ (n,q)=1}}
\prod_{i \in \mathcal{I}} \Lambda_0(n+d_i;q) 
\prod_{i \in \mathcal{H}}\Lambda(n+d_i;q)\\
=| \{ 1 \leq n \leq x \,: \, (n,q)=1 \} |
 \sum_{\mathcal{I} \subseteq \mathcal{J} \subseteq \mathcal{K}} \left( \frac{q}{\varphi(q)} \right)^{k-j} \mathfrak{S}_0 \left( \mathcal{D}_{\mathcal{J}} \right) +O \left(\left(\frac{q}{\varphi(q)}\right)^h N^{1-\delta} \right),
\end{gather*}
where $\mathcal{D}_{\mathcal{J}}=\{ d_i \,:\, i \in \mathcal{J} \}.$ 
Let us denote 
\begin{align*}
a_n:=\prod_{i \in \mathcal{I}} \Lambda_0(n+d_i;q) 
\prod_{i \in \mathcal{H}}\Lambda(n+d_i;q)
\end{align*}
 and 
 \begin{align*}
 c:= \sum_{\mathcal{I} \subseteq \mathcal{J} \subseteq \mathcal{K}} \left( \frac{q}{\varphi(q)} \right)^{k-j} \mathfrak{S}_0 \left( \mathcal{D}_{\mathcal{J}} \right),
 \end{align*}
so that the balanced sum
\begin{align} \label{eq:xxet}
\widetilde{A}(x):=\sum_{\substack{1 \leq n \leq x \\ (n,q)=1}}
(a_n-c) \ll \left(\frac{q}{\varphi(q)}\right)^h N^{1-\delta}.
\end{align}
We also denote 
\begin{align*}
f(n):=\prod_{i \in \mathcal{H}}
(\log (n+d_i))^{m_i-1}
\left( \log (n+d_i)-\frac{q}{\varphi(q)} \right).
\end{align*}
Then 
\begin{align*}
\sum_{\substack{1 \leq n \leq N-d_k\\(n,q)=1}} a_nf(n)=
c\sum_{\substack{1 \leq n \leq N-d_k\\(n,q)=1}} f(n)+
\sum_{\substack{1 \leq n \leq N-d_k\\(n,q)=1}} (a_n-c)f(n).
\end{align*}
By partial summation, the last sum is
\begin{align*}
\int_{1}^{N-d_k} f(x) d\widetilde{A}(x) 
&= \widetilde{A}(N-d_k)f(N-d_k) -
\int_1^{N-d_k} \widetilde{A}(x) f'(x) dx.
\end{align*}
Therefore, using (\ref{eq:xxet}), we have
\begin{align} \label{eq:anfn}
\sum_{\substack{1 \leq n \leq N-d_k\\(n,q)=1}} a_nf(n)=
c\sum_{\substack{1 \leq n \leq N-d_k\\(n,q)=1}} f(n)
+O\left( \left(\frac{q}{\varphi(q)}\right)^h N^{1-\delta}(\log N)^K \right).
\end{align}
For $m>0,$ $\Lambda_m(n;q)$ is nonzero when $n$ is a prime-power, and
\begin{align*}
\Lambda_m(n;q)=\Lambda(n)(\log n)^{m-1}\left(\log n -\frac{q}{\varphi(q)}\right)
\end{align*}
if $n$ is a prime. Thus if $n$ is an integer such that $n+d_i$ is prime for all $i \in \mathcal{H},$ then $a_nf(n)=\prod_{i=1}^k \Lambda_{m_i}(n+d_i;q).$ Also, since the contribution from those $n$ for which $n+d_i$ is a higher power of a prime for one or more $i \in \mathcal{H}$ is $\ll N^{\frac{1}{2}+\epsilon}$ to (\ref{eq:anfn}), we have
\begin{align*}
\sum_{\substack{1 \leq n \leq N-d_k\\(n,q)=1}} \prod_{i=1}^k \Lambda_{m_i}(n+d_i;q)
= c
\sum_{\substack{1 \leq n\leq N-d_k\\(n,q)=1}} 
f(n) 
+O\left( \left(\frac{q}{\varphi(q)}\right)^h N^{1-\delta}(\log N)^K \right).
\end{align*}
Inserting into (\ref{eq:aftercovii}) followed by reversing the change of variables $n_i=n+d_i$ gives $L_k(\boldsymbol{m};q)$
\begin{gather} 
= \frac{1}{k!}   \sum_{\mathcal{I} \subseteq \mathcal{J} \subseteq \mathcal{K}} \left( \frac{q}{\varphi(q)} \right)^{k-j} \sideset{}{^*}\sum_{a \Mod{q}} 
\sum_{\substack{1 \leq n_i \leq N\\n_i \equiv a \Mod{q} \\ n_i \text{ distinct}}}
\mathfrak{S}_0 \left( \mathcal{N}_{\mathcal{J}} \right)
\prod_{i \in \mathcal{H}}
(\log n_i)^{m_i-1}
\left( \log n_i-\frac{q}{\varphi(q)} \right) \nonumber \\
+O\left( \left(\frac{N}{\varphi(q)}\right)^k N^{1-\delta}(\log N)^K\right). \label{eq:weaketii}
\end{gather}
By definition, the main term here is
\begin{align} \label{eq:mtlkii}
\frac{1}{k!} \sum_{\mathcal{I} \subseteq \mathcal{J} \subseteq \mathcal{K}} \left( \frac{q}{\varphi(q)} \right)^{k-j} \sideset{}{^*}\sum_{a \Mod{q}} 
R_{j}(N;q,a)
\sum_{\substack{1 \leq n_i \leq N, i \notin \mathcal{J}\\n_i \equiv a \Mod{q} \\ n_i \text{ distinct}}}
\prod_{i \in \mathcal{H}}
(\log n_i)^{m_i-1}
\left( \log n_i-\frac{q}{\varphi(q)} \right).
\end{align}
For convenience, let us denote
\begin{align*}
S_{\boldsymbol{m}}(N;q,a):=
 \sum_{\substack{1 \leq n_i \leq N, i \in \mathcal{H}\\n_i \equiv a \Mod{q} \\ n_i \text{ distinct}}}
\prod_{i \in \mathcal{H}}
(\log n_i)^{m_i-1}
\left( \log n_i-\frac{q}{\varphi(q)} \right).
\end{align*}
Then, isolating the term $\mathcal{J}=\mathcal{I},$ it follows from Lemma \ref{lem:expsumii} and Lemma \ref{lem:sumofsingii} that (\ref{eq:mtlkii}) is
\begin{gather} 
\frac{1}{k!} \left( \frac{q}{\varphi(q)} \right)^{h} 
 \sideset{}{^*}\sum_{a \Mod{q}}
R_{k-h}(N;q,a)
S_{\boldsymbol{m}}(N;q,a) \nonumber\\
+O \left( \varphi(q) (\log N)^{\sum_{i \in \mathcal{H}} m_i}
\sum_{k-h+1 \leq j \leq k}
\left(\frac{N}{\varphi(q)} \right)^{j/2} \left(\log \frac{N}{q} + (\log \log q)^3 \right)^{j/2} 
\left( \frac{N}{\varphi(q)} \right)^{k-j}
\right) \nonumber\\
= \frac{1}{k!}  \left( \frac{q}{\varphi(q)} \right)^{h}  \sideset{}{^*}\sum_{a \Mod{q}}
R_{k-h}(N;q,a)
S_{\boldsymbol{m}}(N;q,a) \nonumber\\
+
O\left( \varphi(q) (\log N)^{K-k} \left(\frac{N}{\varphi(q)} \right)^{(k-h+1)/2} \left(\log \frac{N}{q}+(\log \log q)^3 \right)^{(k-h+1)/2} 
 \left( \frac{N}{\varphi(q)} \right)^{h-1} \right), \label{eq:RSii}
\end{gather}
where we have used the fact that
\begin{align*}
\sum_{i \in \mathcal{H}}m_i = \sum_{i=1}^k m_i
\leq \sum_{i=1}^k (M_i-1) =K-k.
\end{align*}

Before we proceed, note that
\begin{align} \label{eq:K>h+k}
K=\sum_{i=1}^k M_i =
\sum_{i \in \mathcal{H}} M_i 
+\sum_{i \in \mathcal{J}} M_i 
\geq 2|\mathcal{H}|+|\mathcal{J}| =2h+(k-h)=h+k.
\end{align}
First, we show that the contribution to (\ref{eq:b4deflkii}) from those terms for which $h+k<K$ is negligible.  Since 
\begin{align*}
R_{k-h}(N;q,a) \ll \left( \frac{N}{\varphi(q)} \right)^{(k-h)/2}  \left( 
 \log \frac{N}{q}+(\log \log q)^3 \right)^{(k-h)/2}
\end{align*}
 by Lemma \ref{lem:expsumii} and Lemma \ref{lem:sumofsingii}, it follows from (\ref{eq:weaketii}), (\ref{eq:RSii}) and (\ref{eq:K>h+k}) that 
\begin{gather*}
L_{k}(\boldsymbol{m};q) \ll \varphi(q) (\log N)^{K-k}
\left(\frac{N}{\varphi(q)} \right)^{(k-h)/2}
\left( \log \frac{N}{q} + (\log\log q)^3 \right)^{(k-h)/2}
\left( \frac{N}{\varphi(q)} \right)^{h} \\
+\left(\frac{N}{\varphi(q)}\right)^k N^{1-\delta}(\log N)^K \\
\ll \varphi(q)  (\log N)^K \left( \frac{N}{\varphi(q)\log N} \right)^{(k+h)/2}
\left( \frac{\log (N/q)+(\log \log q)^3}{\log N} \right)^{(k-h)/2} \\
+\left(\frac{N}{\varphi(q)}\right)^k N^{1-\delta}(\log N)^K.
\end{gather*}
Therefore, using (\ref{eq:K>h+k}), the contribution of these terms to (\ref{eq:b4deflkii}) is
\begin{align} \label{eq:etK>h+kii}
\ll \varphi(q) (\log N)^K \left( \frac{N}{\varphi(q)\log N} \right)^{(K-1)/2}+\left(\frac{N}{\varphi(q)}\right)^K N^{1-\delta}(\log N)^K.
\end{align}

Finally, we consider those terms in (\ref{eq:b4deflkii}) for which $h+k=K.$ Since $h \leq k,$ it follows that $k \geq K/2.$ We also have $h=|\mathcal{H}|=K-k,$ and $j=|\mathcal{I}|=k-h=2k-K.$ Since equality holds in (\ref{eq:K>h+k}), it follows that $M_i=2$ for all $i \in \mathcal{H}$ and that $M_i=1$ for all $i \in \mathcal{I}.$ Thus $m_i=M_i-1$ for all $i,$ and for such $\boldsymbol{m}$ we have 
\begin{align*}
S_{\boldsymbol{m}}(N;q,a)&=
\sum_{\substack{1 \leq n_i \leq N, i \in \mathcal{H}\\n_i \equiv a \Mod{q} \\ n_i \text{ distinct}}} \prod_{i \in \mathcal{H}}\left( \log n_i -\frac{q}{\varphi(q)} \right) \\
&= \left( \sum_{\substack{1 \leq n \leq N\\n \equiv a \Mod{q} \\}} \left( \log n -\frac{q}{\varphi(q)} \right) \right)^{K-k}+O \left(
\left( \frac{N}{q} \right)^{K-k-1}(\log N)^{K-k}
\right)\\
&=: S(N;q,a)^{K-k}+O \left(
\left( \frac{N}{q} \right)^{K-k-1}(\log N)^{K-k}
\right).
\end{align*}
Again, since 
\begin{gather*}
R_{2k-K}(N;q,a) \ll \left( \frac{N}{\varphi(q) } \left(\log \frac{N}{q} +(\log \log q)^3 \right) \right)^{(2k-K)/2}
\end{gather*}
by Lemma \ref{lem:expsumii} and Lemma \ref{lem:sumofsingii}, it follows from (\ref{eq:weaketii}) and (\ref{eq:RSii}) that
\begin{align*}
L_k(\boldsymbol{m};q)&=\frac{1}{k!} 
\left( \frac{q}{\varphi(q)} \right)^{K-k}
\sideset{}{^*}\sum_{a \Mod{q}} 
R_{2k-K}(N;q,a)S(N;q,a)^{K-k}\\
&+O\left( \varphi(q) (\log N)^{K-k} \left(\frac{N}{\varphi(q)} \left(\log \frac{N}{q}+(\log \log q)^3 \right) \right)^{(2k-K+1)/2} 
 \left( \frac{N}{\varphi(q)} \right)^{K-k-1} \right)\\
&+O\left( \left(\frac{N}{\varphi(q)}\right)^k N^{1-\delta}(\log N)^K\right) \\
&=\frac{1}{k!} \left( \frac{q}{\varphi(q)} \right)^{K-k}
\sideset{}{^*}\sum_{a \Mod{q}} 
R_{2k-K}(N;q,a)S(N;q,a)^{K-k}\\
&+O\left( \varphi(q) (\log N)^{K} \left(\frac{N}{\varphi(q) \log N}  \right)^{(K-1)/2}
 \right)+O\left( \left(\frac{N}{\varphi(q)}\right)^K N^{1-\delta}(\log N)^K\right).
\end{align*}
Once $k$ is selected, there are precisely ${k \choose K-k}$ ways of choosing the set $\mathcal{H},$ and hence using (\ref{eq:etK>h+kii}) and the above, we have
\begin{gather*}
\sideset{}{^*}\sum_{a \Mod{q}}
\left( \sum_{\substack{1 \leq n \leq N\\n\equiv a \Mod{q}}} \left(\Lambda(n)-\frac{q}{\varphi(q)}  \right) \right)^K \\
=
\sum_{K/2 \leq k \leq K} \frac{K!}{k! 2^{K-k}}
{k \choose K-k} 
\left( \frac{q}{\varphi(q)} \right)^{K-k}
\sideset{}{^*}\sum_{a \Mod{q}} 
R_{2k-K}(N;q,a)S(N;q,a)^{K-k} \\ 
+O\left( \varphi(q) (\log N)^{K} \left(\frac{N}{\varphi(q) \log N}  \right)^{(K-1)/2}
 \right)+O\left( \left(\frac{N}{\varphi(q)}\right)^K N^{1-\delta}(\log N)^K\right).
\end{gather*}

Suppose $K$ is odd. Then so is $2k-K,$ and hence by Lemma \ref{lem:expsumii} and Lemma \ref{lem:sumofsingii}, the main term here is
\begin{align*}
\ll \varphi(q) (\log N)^{\frac{K}{2}}
\left( \frac{N}{\varphi(q)} \right)^{\frac{K}{2}-\frac{1}{7K}+\epsilon},
\end{align*}
so that
\begin{gather*}
\sideset{}{^*}\sum_{a \Mod{q}}
\left( \sum_{\substack{1 \leq n \leq N\\n\equiv a \Mod{q}}} \left(\Lambda(n)-\frac{q}{\varphi(q)}  \right) \right)^K \\
 \ll 
\varphi(q) 
\left( \frac{N \log N}{\varphi(q)} \right)^{\frac{K}{2}}
\left( \frac{N}{\varphi(q)\log N} \right)^{-\frac{1}{8K}}
+\left(\frac{N}{\varphi(q)}\right)^K N^{1-\delta}(\log N)^K.
\end{gather*}

Suppose $K$ is even. Then, it follows from Lemma \ref{lem:sumofsingii} that
\begin{gather*}
\sideset{}{^*}\sum_{a \Mod{q}}
\left( \sum_{\substack{1 \leq n \leq N\\n\equiv a \Mod{q}}} \left(\Lambda(n)-\frac{q}{\varphi(q)}  \right) \right)^K \\
=\sum_{k=K/2}^K
\frac{K!}{k!2^{K-k}} {k \choose K-k} 
\left( \frac{q}{\varphi(q)} \right)^{K-k} \\
\times \mu_{2k-K}
\left( \left(\frac{\varphi(q)}{q}\right)^{-2} \left(
V_2(Q;q,a)- H \sum_{\substack{1<r \mid Q }} \frac{\mu^2 (r)}{\varphi(r)}
\right) \right)^{k-\frac{K}{2}} 
  \sideset{}{^*}\sum_{a \Mod{q}} S(N;q,a)^{K-k} \\
+O\left(\varphi(q) 
\left( \frac{N \log N}{\varphi(q)} \right)^{\frac{K}{2}}
\left( \frac{N}{\varphi(q)\log N} \right)^{-\frac{1}{8K}}
+ \left(\frac{N}{\varphi(q)}\right)^K N^{1-\delta}(\log N)^K\right).
\end{gather*}
Making the change of variables $k=\frac{K}{2}+l,$ 
the main term here becomes
\begin{gather*}
\mu_K \sideset{}{^*}\sum_{a \Mod{q}} \sum_{l=0}^{K/2} 
{K/2 \choose l} \left( \frac{q}{\varphi(q)}\sum_{\substack{1 \leq n \leq N\\n \equiv a \Mod{q}}} \left( \log n -\frac{q}{\varphi(q)} \right) \right)^{\frac{K}{2}-l} \\
\times  \left(  \left(\frac{\varphi(q)}{q}\right)^{-2} \left(
V_2(Q;q,a)- H \sum_{\substack{1<r \mid Q }} \frac{\mu^2 (r)}{\varphi(r)}
\right) \right)^{l},
\end{gather*}
which is by the binomial theorem
\begin{align*}
 \mu_K  \sideset{}{^*}\sum_{a \Mod{q}}
\left( 
 \frac{q}{\varphi(q)}\sum_{\substack{1 \leq n \leq N\\n \equiv a \Mod{q}}} \left( \log n -\frac{q}{\varphi(q)} \right)
 +
  \left(\frac{\varphi(q)}{q}\right)^{-2} \left(
V_2(Q;q,a)- H \sum_{\substack{1<r \mid Q }} \frac{\mu^2 (r)}{\varphi(r)}
\right)
\right)^{\frac{K}{2}}.
\end{align*}
By partial summation, we have
\begin{align*}
\sum_{\substack{1 \leq n \leq N\\n \equiv a \Mod{q}}}  \log n = \frac{N}{q}\log N -\frac{N}{q}+O(\log N).
\end{align*}
Finally, the theorem follows readily from Lemma \ref{lem:expsumii}.

\section{Proof of Theorem \ref{thm:poisson}}

We need an estimate for the sum of singular series along an arithmetic progression.

\begin{lemma} \label{lem:sum}
Let $k \geq 1$ and $(a,q)=1.$ Then for any $N \geq 2q,$ we have
\begin{align*}
\sum_{\substack{1 \leq n_i\leq N\\n_i \equiv a \Mod{q}\\ n_i \text{ \normalfont distinct}}}
\mathfrak{S}(\mathcal{N};q)=\left( \frac{N}{\varphi(q)} \right)^k
+O \left( \left( \frac{N}{\varphi(q)} \right)^{k-1} 
\log \frac{N}{q}  \right),
\end{align*}
where $\mathcal{N}=\{ n_1,\ldots,n_k \}.$
\end{lemma}

\begin{proof}
By definition, we have
\begin{align*}
\mathfrak{S}(\mathcal{N};q)=
\sum_{\mathcal{M}\subseteq \mathcal{N}}
\left( \frac{q}{\varphi(q)} \right)^{k-|\mathcal{M}|}
\mathfrak{S}_0(\mathcal{M};q).\
\end{align*}
Let $R_{k}(N;q,a)$ denote the sum
\begin{align*}
\sum_{\substack{1 \leq n_{i} \leq N
\\ n_{i} \equiv a \Mod{q}\\ n_i \text{ \normalfont distinct}}}
\mathfrak{S}_0\left( \mathcal{N};q \right).
\end{align*}
Then
\begin{gather*}
\sum_{\substack{1 \leq n_i\leq N\\n_i \equiv a \Mod{q}\\ n_i \text{ \normalfont distinct}}}
\mathfrak{S}(\mathcal{N};q)=
\sum_{\mathcal{M}\subseteq \mathcal{N}}
\left( \frac{q}{\varphi(q)} \right)^{k-|\mathcal{M}|}
\prod_{j=|\mathcal{M}|}^{k-1}
\left(H-j\right)
\sum_{\substack{1 \leq m_i\leq N\\m_i \equiv a \Mod{q}\\m_i \text{ distinct}}}
\mathfrak{S}_0(\mathcal{M};q)\\
=\left( \frac{q}{\varphi(q)} \right)^{k} H^k+
{k \choose 2}
\left( \frac{q}{\varphi(q)} \right)^{k-2} 
H^{k-2}R_2(N;q,a) +O \left( \frac{q}{\varphi(q)}
H^{k-1}
\right)\\
+O \left( \sum_{j=3}^k \left(\frac{q}{\varphi(q)}\right)^{k-j}
H^{k-j} |R_{j}(N;q,a)|  \right)
,
\end{gather*}
where $H:=|\{ 1 \leq n \leq N \,: \, n \equiv a \Mod{q} \}|.$
Then, the lemma follows readily from Lemma \ref{lem:expsumii} and Lemma \ref{lem:sumofsingii}.
\end{proof}

\begin{proof}[Proof of Theorem \ref{thm:poisson}]
When $k=1,$ one can verify easily using the prime number theorem. Therefore, we can assume $k \geq 2.$ By the definition of Stirling numbers of the second kind, we have
\begin{align*}
\frac{1}{\varphi(q)}\sideset{}{^*}\sum_{a \Mod{q}}
\pi(N;q,a)^k
&= \frac{1}{\varphi(q)}\sideset{}{^*}\sum_{a \Mod{q}}
 \sum_{\substack{1 \leq p_1,\ldots,p_k\leq N\\p_1 \equiv \cdots \equiv p_k \equiv a \Mod{q}}}1 \\
&=\frac{1}{\varphi(q)}\sideset{}{^*}\sum_{a \Mod{q}}
\sum_{j=1}^k \stirling{k}{j} j!
\sum_{\substack{1 \leq p_1<\cdots < p_j \leq N\\p_1 \equiv \cdots \equiv p_j \equiv a \Mod{q}}}1.
\end{align*}
Making the change of variables $p_i=n+d_i$ for $i=1,\ldots,j,$ this becomes
\begin{gather} 
\frac{1}{\varphi(q)}\sideset{}{^*}\sum_{a \Mod{q}}\sum_{j=1}^k \stirling{k}{j} j!
 \sum_{\substack{0=d_1 <\cdots <d_j < N\\d_i \equiv 0 \Mod{q}\\ }}
\sum_{\substack{n\,:\, 1 \leq n+d_i \leq N\\ n \equiv a \Mod{q}}}
\prod_{i=1}^j 1_{\mathbb{P}}(n+d_i) \nonumber\\
=\frac{1}{\varphi(q)}\sum_{j=1}^k \stirling{k}{j} j!
 \sum_{\substack{0=d_1 <\cdots <d_j < N\\d_i \equiv 0 \Mod{q}\\ }}
\sum_{\substack{n\,:\, 1 \leq n+d_i \leq N\\ (n,q)=1}}
\prod_{i=1}^j 1_{\mathbb{P}}(n+d_i) \nonumber\\
=\frac{1}{\varphi(q)}\sum_{j=1}^k \stirling{k}{j} j!
\sum_{\substack{0=d_1 <\cdots <d_j < N\\d_i \equiv 0 \Mod{q}\\ }}
\sum_{\substack{1 \leq n \leq N-d_j\\ (n,q)=1}}
\prod_{i=1}^j 1_{\mathbb{P}}(n+d_i) \nonumber\\
=\frac{1}{\varphi(q)}\sum_{\substack{1 \leq n \leq N\\(n,q)=1}}1_{\mathbb{P}}(n)
+\Sigma_1+\Sigma_2, \label{eq:3sums}
\end{gather}
where
\begin{align*}
\Sigma_1:=
\frac{1}{\varphi(q)}\sum_{j=2}^k \stirling{k}{j} j!
\sum_{\substack{0=d_1 <\cdots <d_j < N-q\\d_i \equiv 0 \Mod{q}\\ }}
\sum_{\substack{1 \leq n \leq N-d_j\\ (n,q)=1}}
\prod_{i=1}^j 1_{\mathbb{P}}(n+d_i)
\end{align*}
and $\Sigma_2$ is the sum of the remaining terms, i.e., $0=d_1<\cdots<d_{j-1}<N-q \leq d_j.$

Let us deal with $\Sigma_2$ first. Since $1 \leq N-d_j \leq q,$ assuming Conjecture \ref{conj:ii}, the sum $\Sigma_2$ is
\begin{align*}
\ll \frac{q}{\varphi(q)} \times \max_{2 \leq j \leq k} \frac{1}{\log^j q}
\sum_{\substack{0=d_1 <\cdots <d_{j-1} < N-q\\d_i \equiv 0 \Mod{q}\\ }}
\mathfrak{S}(\mathcal{D};q)+\frac{1}{\varphi(q)}
\left( \frac{N}{q} \right)^{k-1} N^{1-\delta},
\end{align*}
where $\mathcal{D}=\{ d_1, \ldots, d_j \}.$
Using (\ref{eq:usefulater}) with $y=N/q,$ i.e., $\mathfrak{S}(\mathcal{D};q) \ll (\varphi(q)/q)^{-k}(\log (N/q))^{k-1},$ and the assumption $\varphi(q) \asymp N/\log N,$ the first term is
\begin{align} \label{eq:sigma2}
\ll \left( \frac{\varphi(q)}{q} \right)^{-(k+1)} \left( \log \frac{N}{q} \right)^{k-1} \times \frac{1}{\log^2 q}  \ll \frac{(\log \log N)^{2k}}{\log^2 N}.
\end{align}
Also, the second term is $\ll N^{-\delta/2},$ which is negligible.

It remains to estimate $\Sigma_1.$
Since $N/\log N \ll q \leq N-d_j \leq N$ for any $\mathcal{D},$ assuming Conjecture \ref{conj:ii}, the sum $\Sigma_1$ is
\begin{align*} 
\frac{1}{\varphi(q)}\sum_{j=2}^k \stirling{k}{j} j!
\sum_{\substack{0=d_1 <\cdots <d_j < N-q\\d_i \equiv 0 \Mod{q}\\ }}
 \left( 1+O \left( \frac{1}{\log N} \right) \right)
\frac{\mathfrak{S}(\mathcal{D};q)}{\log^j N} \sum_{\substack{1 \leq n \leq N-d_j\\(n,q)=1}}1+O ( N^{-\frac{\delta}{2}} ).
\end{align*}
Arguing as above, this is
\begin{gather} 
\frac{1}{\varphi(q)}\sum_{j=2}^k \stirling{k}{j} j!
\sum_{\substack{0=d_1 <\cdots <d_j < N-q\\d_i \equiv 0 \Mod{q}\\ }}
 \left( 1+O \left( \frac{1}{\log N} \right) \right)
\frac{\mathfrak{S}(\mathcal{D};q)}{\log^j N} \sum_{\substack{1 \leq n \leq N-q-d_j\\(n,q)=1}}1 \nonumber\\
+O \left( \frac{(\log \log N)^{2k}}{\log^2 N} \right)+O(N^{-\frac{\delta}{2}}). \label{eq:detdet}
\end{gather}

Finally, since 
\begin{align*}
\frac{1}{\varphi(q)}\sum_{\substack{1 \leq n \leq N\\(n,q)=1}}1_{\mathbb{P}}(n)
=\frac{1}{\varphi(q)}\left(1+O \left( \frac{1}{\log N} \right) \right)
\frac{1}{\log N}
\sum_{\substack{1 \leq n \leq N-q\\(n,q)=1}} 1
\end{align*}
by the prime number theorem, combining with (\ref{eq:3sums}), (\ref{eq:sigma2}) and (\ref{eq:detdet}) yields
\begin{gather*}
\frac{1}{\varphi(q)}\sideset{}{^*}\sum_{a \Mod{q}}
\pi(N;q,a)^k \\
=
\frac{1}{\varphi(q)}\sum_{j=1}^k \stirling{k}{j} j!
\sum_{\substack{0=d_1 <\cdots <d_j < N-q\\d_i \equiv 0 \Mod{q}\\ }}
 \left( 1+O \left( \frac{1}{\log N} \right) \right)
\frac{\mathfrak{S}(\mathcal{D};q)}{\log^j N} \sum_{\substack{1 \leq n \leq N-q-d_j\\(n,q)=1}}1 \\
+O \left( \frac{(\log \log N)^{2k}}{\log^2 N} \right)+O(N^{-\frac{\delta}{2}}).
\end{gather*}
Note that the error terms here are negligible. Making the change of variables $n_i=n+d_i$ for $i=1,\ldots,j,$ followed by separating into residue classes, the main term becomes
\begin{gather*}
\frac{1}{\varphi(q)}\sum_{j=1}^k\stirling{k}{j} 
\left( 1+O \left( \frac{1}{\log N} \right) \right) \frac{1}{\log^j N} 
\sideset{}{^*}\sum_{a \Mod{q}} 
\sum_{\substack{1 \leq n_i\leq N-q\\n_i \equiv a \Mod{q}
\\n_i\text{ distinct}}}
\mathfrak{S}(\mathcal{N};q)
,
\end{gather*}
which is by Lemma \ref{lem:sum}
\begin{gather*}
\left( 1+O \left( \frac{1}{\log N} \right) \right)\sum_{j=1}^k\stirling{k}{j}  
\left(1+O \left( \frac{\varphi(q)}{N}   \log \left(\frac{N}{q} \right) \right) \right)\left( \frac{N-q}{\varphi(q)\log N} \right)^j\\
=\left( 1 +O\left( \frac{\log\log N}{\log N} \right) \right)\sum_{j=1}^k \stirling{k}{j}
\left( \frac{N}{\varphi(q)\log N} \right)^j
\end{gather*}
since again $\varphi(q) \asymp N/\log N,$ and the theorem follows.
\end{proof}

\section{Some statistics on the least primes in arithmetic progressions} \label{sec:stat}

We begin by fixing the bin width to be $0.1.$ For those displaying the ``probability density function (PDF)", the histograms in orange represent ten times the proportion of $\frac{p(q,a)}{\varphi(q)\log q} \in (t,t+0.1],$ while the curves in blue represent the function $y=e^{-t}.$ 

On the other hand, for those displaying the ``cumulative distribution function (CDF)", the histograms in orange represent the proportion of $\frac{p(q,a)}{\varphi(q)\log q} \leq t,$ while the curves in blue represent the function $y=1-e^{-t}.$

We investigate the moduli $q=2023=7 \times 17^2, 5749$ and $30030=2 \times 3 \times 5 \times 7 \times 11 \times 13,$ for which the numbers of reduced residue classes (samples) are $1632, 5748$ and $5760$ respectively.

\FloatBarrier 

\begin{figure}[H]
    \centering
    \begin{subfigure}{0.45\textwidth}
        \includegraphics[width=\linewidth]{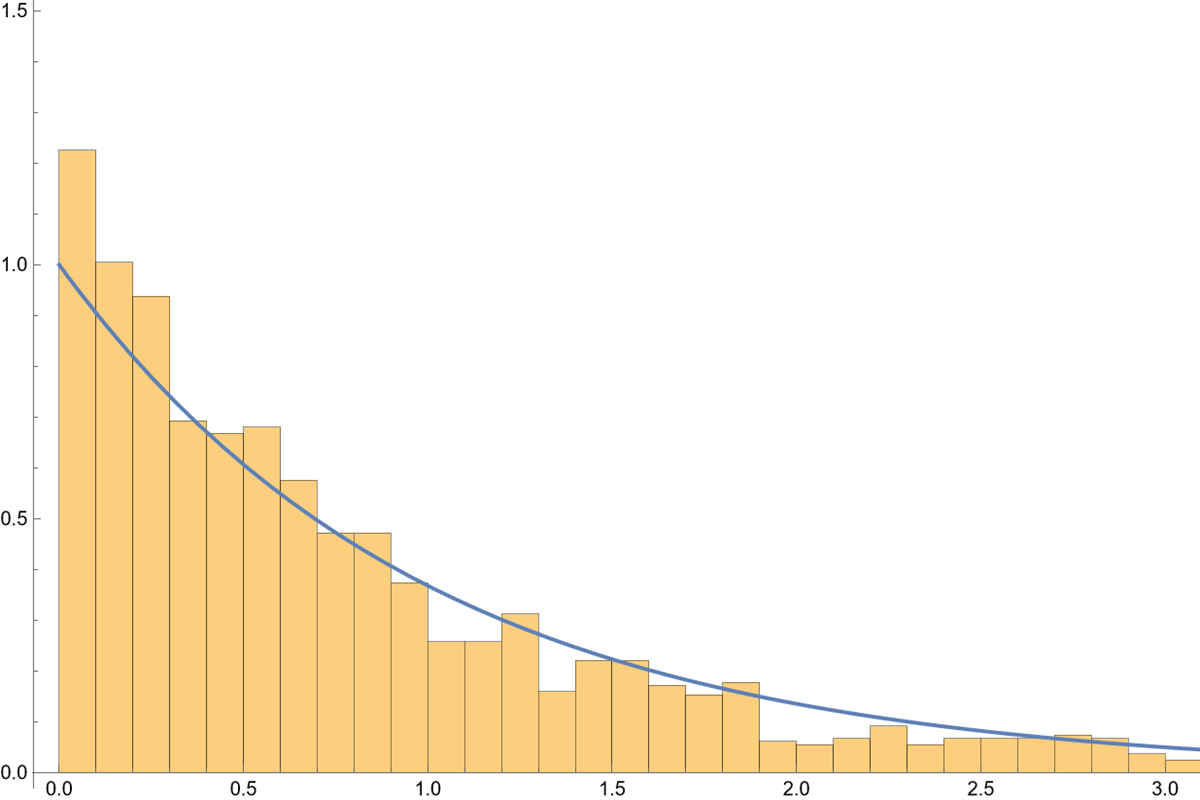}
    \end{subfigure}\hfill
    \begin{subfigure}{0.45\textwidth}
        \includegraphics[width=\linewidth]{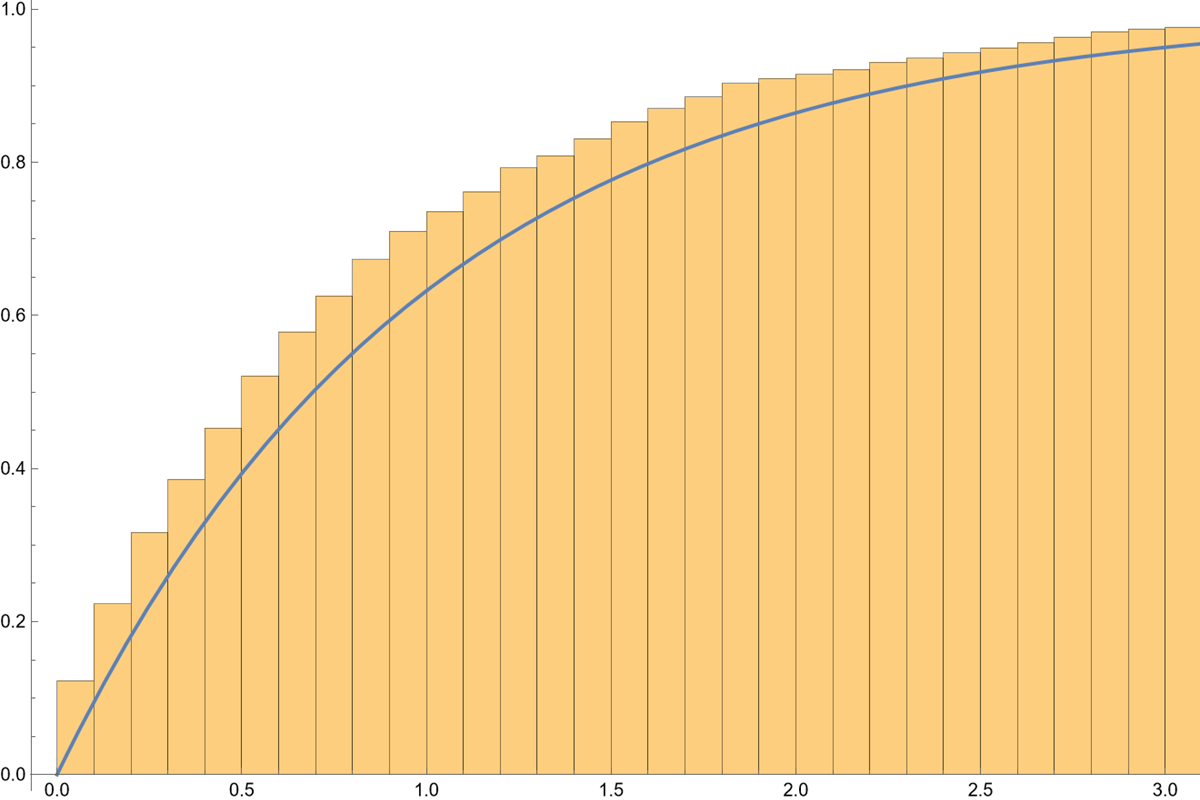}
    \end{subfigure}
    \caption{PDF and CDF when $q=2023$}
\end{figure}

\begin{figure}[H]
    \centering
    \begin{subfigure}{0.45\textwidth}
        \includegraphics[width=\linewidth]{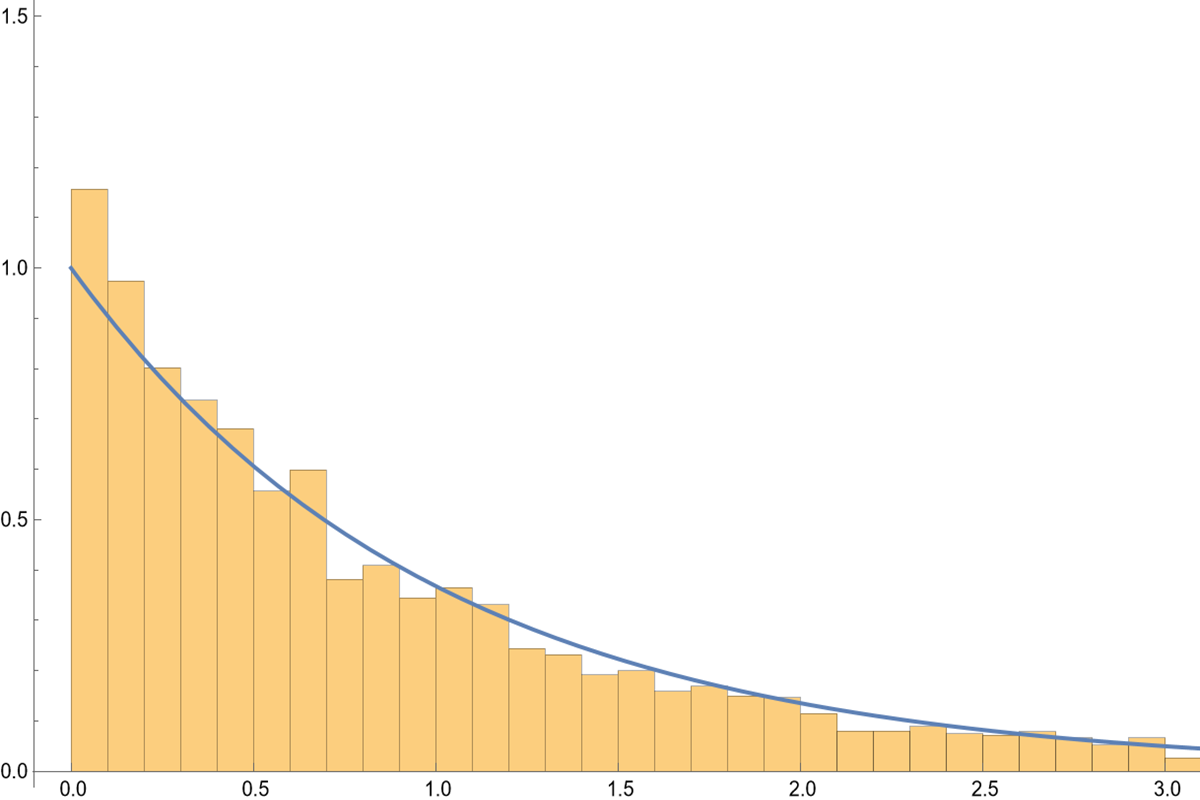}
    \end{subfigure}\hfill
    \begin{subfigure}{0.45\textwidth}
        \includegraphics[width=\linewidth]{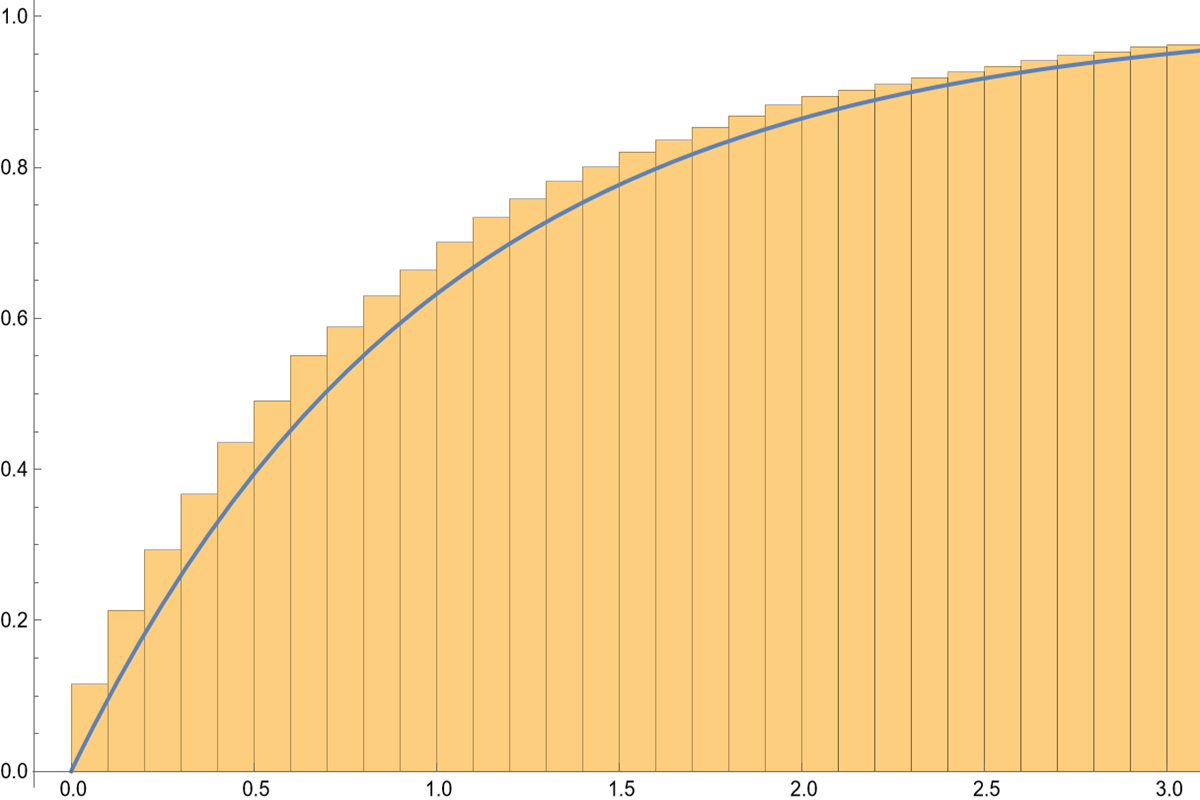}
    \end{subfigure}
    \caption{PDF and CDF when $q=5749$}
\end{figure}

\begin{figure}[H]
    \centering
    \begin{subfigure}{0.48\textwidth}
        \includegraphics[width=\linewidth]{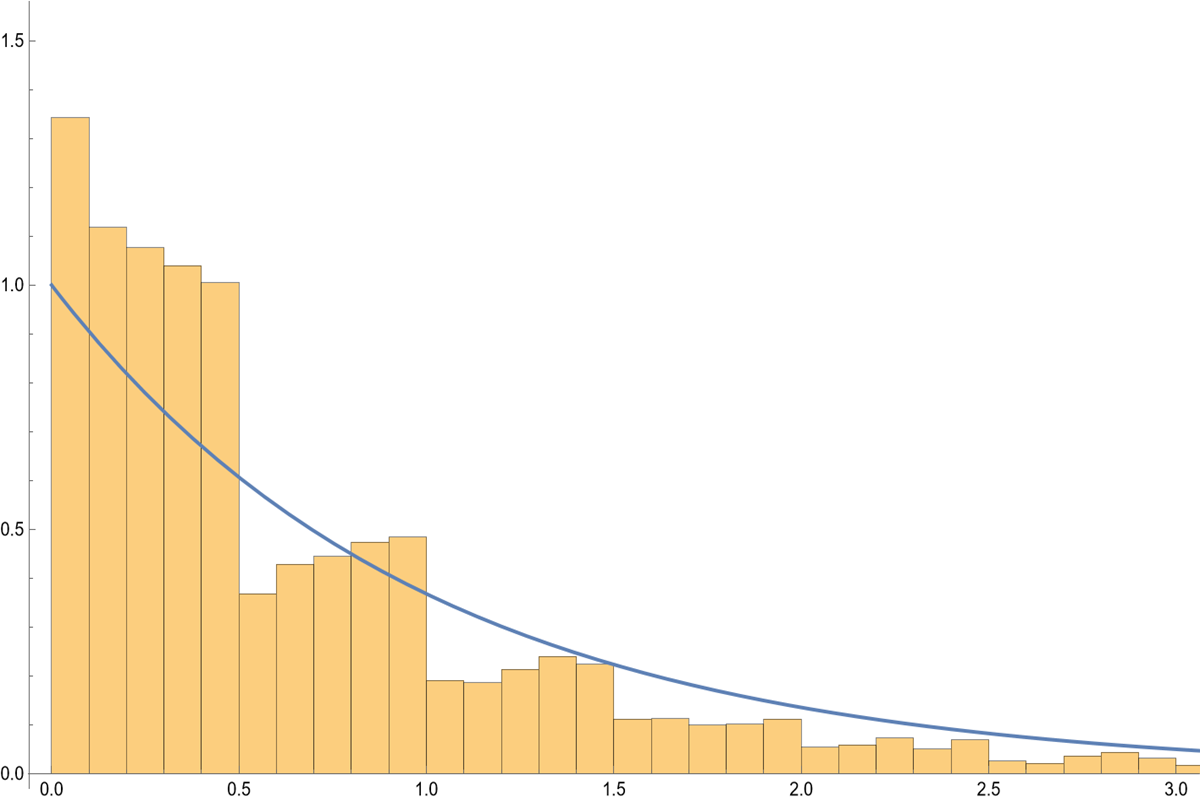}
    \end{subfigure}\hfill
    \begin{subfigure}{0.48\textwidth}
        \includegraphics[width=\linewidth]{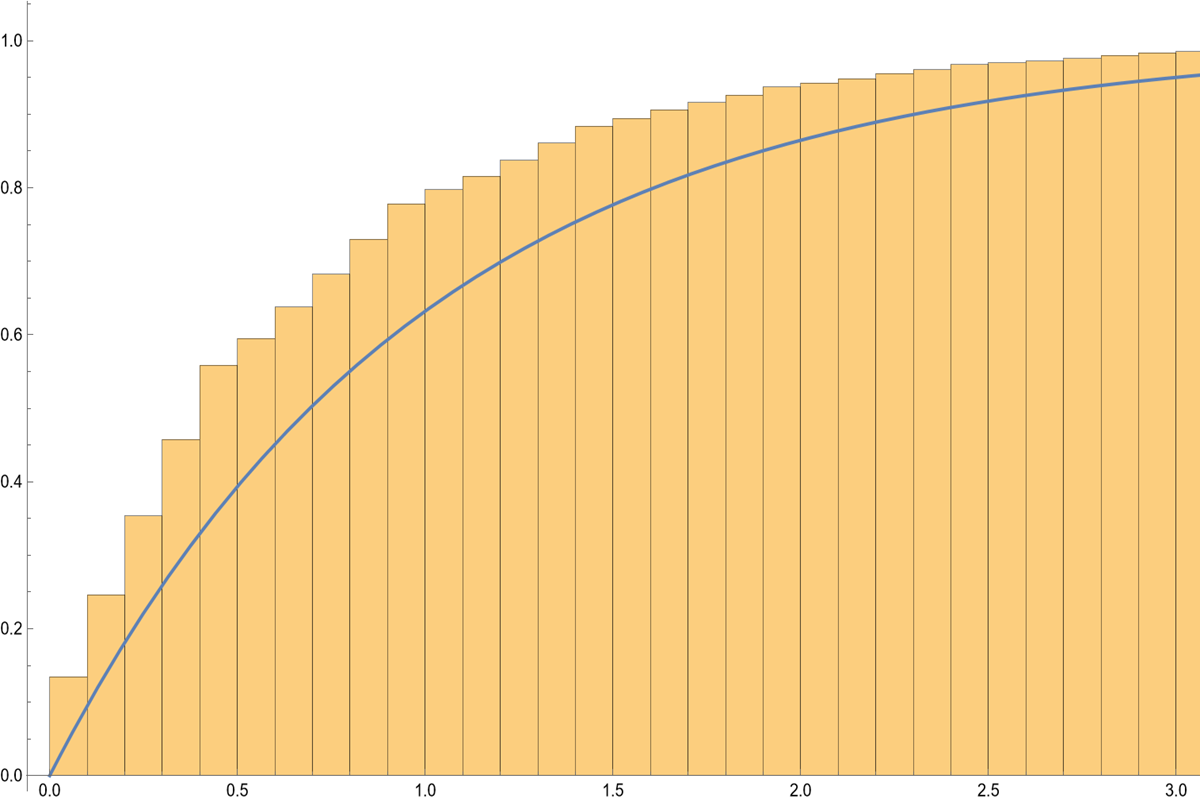}
    \end{subfigure}
    \caption{PDF and CDF when $q=30030$}
\end{figure}

\FloatBarrier 

\section{Unexpected discrepancies} \label{sec:discre}

The final figure raises the question of what accounts for the substantial discrepancies when the modulus $q$ is smooth.

Let $x=t\varphi(q)\log q.$ Then for any $x \leq q \left( \text{or } t \leq \frac{q}{\varphi(q)\log q} \right)$, every prime $\leq x$ must be the least prime in some progressions, so that the proportion of reduced residues $a \Mod{q}$ for which $p(q,a) \leq x=t\varphi(q)\log q$ is 
\begin{gather*}
\sim \frac{\pi(x)}{\varphi(q)} \sim \left( 1+\frac{1}{\log q}\right)t.
\end{gather*}

However, for any $q< x \leq 2q \left( \text{or } \frac{q}{\varphi(q)\log q} <t \leq \frac{2q}{\varphi(q)\log q} \right)$, this is no longer true and under a uniform Hardy--Littlewood conjecture on prime pairs, the proportion of reduced residues becomes
\begin{gather*}
\sim \frac{1}{\varphi(q)}(\pi(x)-\pi_2(x-q;q)) \\
\sim 
\left( 1+\frac{1-\mathfrak{S}(q)}{\log q}\right)\left( t-\frac{q}{\varphi(q)\log q} \right)
+\frac{q}{\varphi(q)\log q}\left(1+\frac{1}{\log q} \right),
\end{gather*}
where 
\begin{align*}
\pi_2(x-q;q):=|\{ p \leq x-q \,:\, p, p+q \text{ are primes}\}|. 
\end{align*}

For \( t \geq 0 \) in general, our modified prediction for the CDF is given by the piecewise continuous linear function \( F \) with a constant slope of 
\[
1 + \frac{1}{\log q} \left( 1 - \sum_{j=1}^k \mathfrak{S}(jq) \right)
\]
on the interval 
\[
\left( \frac{kq}{\varphi(q) \log q}, \frac{(k+1)q}{\varphi(q) \log q} \right]
\]
for each integer \( k \geq 0 \), satisfying the initial condition \( F(0) = 0 \). 

\begin{remark}
Applying Perron's formula (see \cite[Proposition 3]{MR1412558}), it can be verified that for large integers $k$ that do not grow with $q,$ the corresponding slope is
\begin{align*}
\sim 1-\frac{kq}{\varphi(q)\log q} \sim \exp \left( -\frac{kq}{\varphi(q)\log q} \right),
\end{align*}
which matches the slope of the standard exponential distribution's CDF at $t=\frac{kq}{\varphi(q)\log q}.$
\end{remark}

In particular, transitions occur at every multiple of $\frac{q}{\varphi(q)\log q}$, which are noticeable if $q$ is smooth (and not sufficiently large). When $q=30030.$ this quantity is $\approx 0.50568$, which is greater than the bin width $0.1$. Nonetheless, since $\frac{q}{\varphi(q)\log q} \ll \frac{\log \log q}{\log q} \to 0$ as $q \to \infty,$ albeit very slowly, it will eventually be compared with any fixed bin width.

\begin{figure}[h]
    \centering
    \includegraphics[scale=0.45]{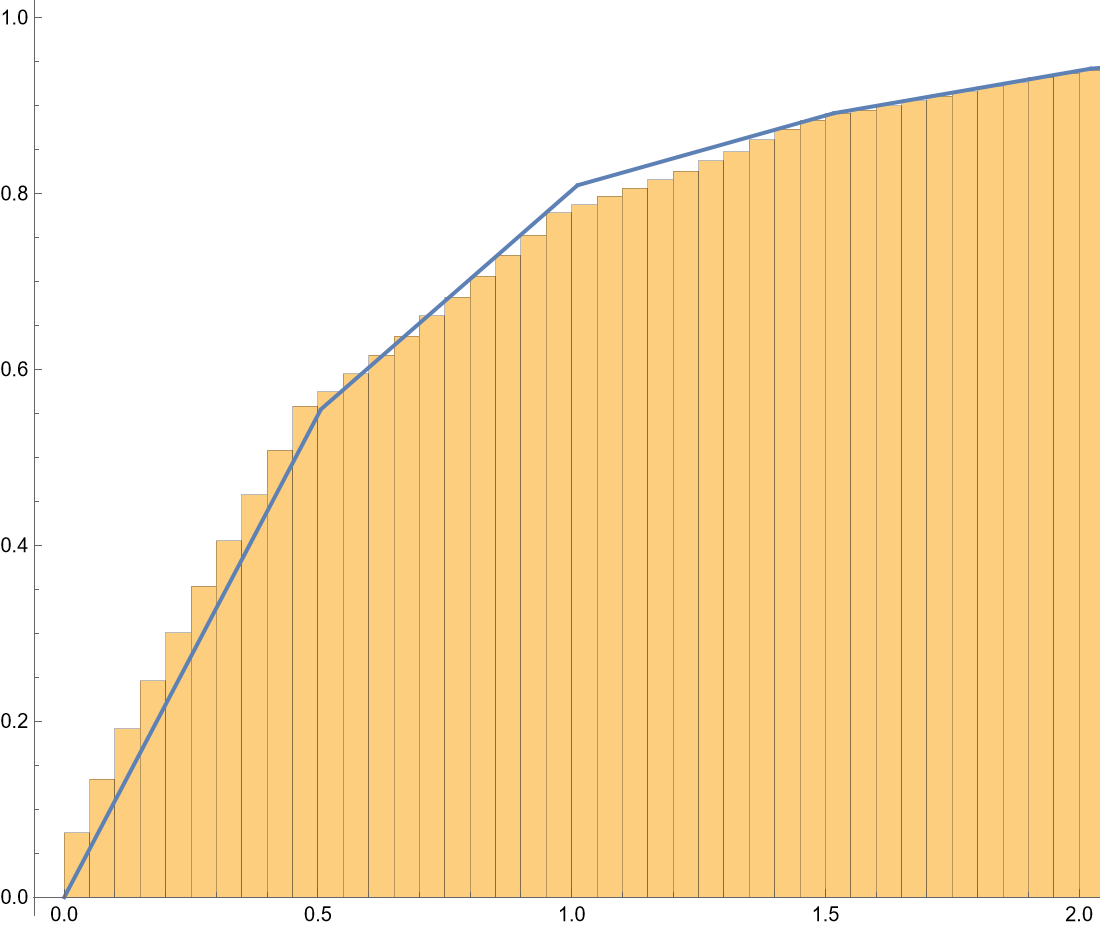}
    \caption{Modified prediction for the CDF when $q=30030$}
\end{figure}




\section{Future work}

Given $m \geq 1$ and $(a,q)=1,$ we are also interested in the $m$-th least prime $p \equiv a \pmod{q},$ which we denote by $p_m(q,a).$ As $q \to \infty,$ not only do we have Corollary \ref{cor:exp}, but we further expect that the normalized primes in progressions 
\begin{gather*}
\frac{p_1(q,a)}{\varphi(q)\log q} < \frac{p_2(q,a)}{\varphi(q)\log q} < \cdots
\end{gather*}
should behave as random points on $(0,\infty)$ as $a \Mod{q}$ varies in the sense of a Poisson point process, which has been explored in \cite{pseudo}.

\section*{Acknowledgements}
The author is grateful to Andrew Granville for his advice and encouragement. He would also like to thank Brad Rodgers, Cihan Sabuncu, and the anonymous referees for their valuable comments. The latter part of this
work was supported by the Swedish Research Council under grant no. 2016-06596 while the author was in residence at Institut Mittag-Leffler in Djursholm, Sweden during the semester of Winter 2024.

\printbibliography

\end{document}